\documentclass{amsart}

\usepackage{amsmath,amsthm, rlepsf}
\usepackage{graphics}
\usepackage{color}
\usepackage{epsfig}
 \usepackage{palatino}

\usepackage{hyperref}

\theoremstyle{plain}
\newtheorem{thm}{Theorem}[section]
\newtheorem{cor}[thm]{Corollary}

\newtheorem{prop}[thm]{Proposition}
\newtheorem{lem}[thm]{Lemma}
\newtheorem{claim}[thm]{Claim}
\newtheorem{quest}{Question}

\theoremstyle{definition}
\newtheorem{remark}[thm]{Remark}

\newcommand{\comment}[1]{}

\DeclareMathOperator{\tor}{tor}
\DeclareMathOperator{\interior}{int}
\newcommand{\N}{\ensuremath{\mathbb{N}}}

\newcommand{\R}{\ensuremath{\mathbb{R}}}
\newcommand{\Z}{\ensuremath{\mathbb{Z}}}

\renewcommand{\sl}{\ensuremath{{\, sl}}}
\DeclareMathOperator{\tb}{tb}
\DeclareMathOperator{\rot}{r}
\def\dfn#1{{\em #1}}

\title{On knots in overtwisted contact structures}

\author{John B. Etnyre}
\address{School of Mathematics \\ Georgia Institute of Technology}
\email{etnyre@math.gatech.edu}
\urladdr{\href{http://www.math.gatech.edu/~etnyre}{http://www.math.gatech.edu/\~{}etnyre}}

\begin{document}

\begin{abstract}
We prove that each overtwisted contact structure has knot types that are represented by infinitely many distinct transverse knots all with the same self-linking number. In some cases, we can even classify all such knots. We also show similar results for Legendrian knots and prove a ``folk'' result concerning loose transverse and Legendrian knots (that is knots with overtwisted complements) which says that such knots are determined by their classical invariants (up to contactomorphism). Finally we discuss how these results partially fill in our understanding of the ``geography" and ``botany'' problems for Legendrian knots in overtwisted contact structures, as well as many open questions regarding these problems. 
\end{abstract}

\maketitle

\section{Introduction}
Since Eliashberg's formative paper \cite{Eliashberg89} classifying overtwisted contact structures on 3--manifolds, the study of and interest in such structures has been minimal. However, in recent years they have been taking a more central role due to their many interesting applications --- such as, the construction of achiral Lefschetz fibrations \cite{EtnyreFuller06} and near symplectic structures \cite{GayKirby04} on certain 4--manifolds and the understanding of the existence of Engel structures on 4--manifolds \cite{Vogel09} --- as well as the interesting knot theory they support. This paper is aimed at studying the Legendrian and transverse knot theory of overtwisted contact structures. We begin with a brief history of the subject.

A Legendrian or transverse knot in an overtwisted contact structure $\xi$ on a 3--manifold $M$ is called \dfn{loose} if the contact structure restricted to its complement is also overtwisted, otherwise the knot is called \dfn{non-loose}. Though apparently known to a few experts, the first explicit example of a non-loose knot was given by Dymara in \cite{Dymara01}, where a single non-loose Legendrian unknot was constructed in a certain overtwisted structure on $S^3$. More recently, Eliashberg and Fraser \cite{EliashbergFraser09}  gave a coarse classification of Legendrian unknots in overtwisted contact structures on $S^3$, see Theorem~\ref{mainS3contacto} below. (We say knots are \dfn{coarsely classified} if they are classified up to co-orientation preserving contactomorphism, smoothly isotopic to the identity. We reserve the word \dfn{classified} to refer to the classification up to Legendrian isotopy, and similarly for transverse knots.) An immediate corollary of this work is that there are no non-loose transverse unknots in any overtwisted contact structure. 

In \cite{Etnyre08} it was shown that there are knot types and overtwisted contact structures for which there were arbitrarily many distinct non-loose Legendrian knots realizing that knot type with fixed Thurston-Bennequin invariant and rotation number. While it is easy to construct non-loose transverse knots in any overtwisted contact structure (one just observes, {\em cf.} \cite{EtnyreVela-Vick10}, that the complement of the binding of a supporting open book decomposition is tight), two non-loose transverse knots with the same self-linking numbers were first produced by Lisca, Ozsv\'ath, Stipsicz and Szab\'o, in \cite{LiscaOzsvathStipsiczSzabo09}, using Heegaard-Floer invariants of Legendrian and transverse knots.

There have been very few results concerning the classification of Legendrian or transverse knots in overtwisted contact structure (as opposed to the coarse classification), but there has been some work giving necessary conditions for the existence of a Legendrian isotopy, see for example \cite{DingGeiges09, Dymara??, EliashbergFraser09}.

Leaving the history of the subject for now we begin by recalling a version of the Bennequin bound for non-loose knots. This result first appeared in \cite{Dymara??} where it was attributed to \'Swi\c{a}tkowski.
\begin{prop}[\'Swi\c{a}tkowski, see \cite{Dymara??}]\label{looseLegbound}
Let $(M,\xi)$ be an overtwisted contact 3--manifold and $L$ a non-loose Legendrian knot in $\xi$. Then 
\[
-|\tb(L)|+|\rot(L)|\leq -\chi(\Sigma)
\]
for any Seifert surface $\Sigma$ for $L$.
\end{prop}
We sketch a simple proof of this result below.  We now observe a relation between non-loose transverse knots and their Legendrian approximations as well as non-loose Legendrian knots and their transverse push-offs.
\begin{prop}\label{nltnll}
If $T$ is a non-loose transverse knot then any Legendrian approximation of $T$ is non-loose.  
If $L$ is a non-loose Legendrian knot then the transverse push-off of $L$ may or may not be non-loose. 
\end{prop}
The previous two results imply that a version of the Bennequin bound for transverse knots in a tight contact structure also holds for non-loose transverse knots. This is in stark contrast to what happens for Legendrian knots as Proposition~\ref{looseLegbound} indicates and Eliashberg and Fraser's coarse classification of unknots, given in Theorem~\ref{mainS3contacto} below, confirms. 

\begin{prop}\label{prop:nlsl}
Let $(M,\xi)$ be a contact 3--manifold and $K$ a transverse knot in $\xi$ with Seifert surface $\Sigma$. If  
\[
\sl(K)>-\chi(\Sigma)
\]
then $K$ is loose (and, of course, $\xi$ is overtwisted). In particular any non-loose knot $K$ in an overtwisted contact structure satisfies the Bennequin inequality
\[
\sl(K)\leq -\chi(\Sigma).
\]
\end{prop}

\subsection{
The coarse classification of loose Legendrian and transverse knots}

The following two theorems make precise the well-known ``folk'' theorems that loose Legendrian or transverse knots are coarsely  classified by their classical invariants. 
\begin{thm}\label{thm:legloose}
Let $(M,\xi)$ be an overtwisted contact manifold. For each null-homologous knot type $\mathcal{K}$ and each pair of integers $(t,r)$ satisfying $t+r$ is odd, there is a unique, up to contactomorphism,  loose Legendrian knot $L$ in the knot type $\mathcal{K}$  with $\tb(L)=t$ and $\rot(L)=r$.
\end{thm}

\noindent
Recall that for any Legendrian knots $L$ we must have $\tb(L)+\rot(L)$ odd, so the above theorem says any possible pair of integers is realized by a unique loose Legendrian knot in any overtwisted contact structure. For transverse knots we have the following result. 

\begin{thm}\label{thm:transverseloose}
Let $(M,\xi)$ be an overtwisted contact manifold. For each null-homologous knot type $\mathcal{K}$ and each odd integer $s$ there is a unique, up to contactomorphism,  loose transverse knot $T$ in the knot type $\mathcal{K}$ with $\sl(T)=s$.
\end{thm}

\noindent
Again recall that the self-linking number of any transverse knot must be odd and thus the theorem says that any possible integer is realized by a unique loose transverse knot in an overtwisted contact structure. 

These two theorems follow directly from Eliashberg's classification of overtwisted contact structures and a careful analysis of homotopy classes of plane fields on manifolds with boundary, which we give in Section~\ref{sec:loose}. Theorem~\ref{thm:legloose} also appears in \cite{EliashbergFraser09} though the details of the homotopy theory were not discussed, and while these details are fairly straight forward they do not seem obvious to the author. In particular, when studying the homotopy classes of plane fields there is both a 2 and 3-dimensional obstruction to being able to homotope one plane field to another. The fact that the 3-dimensional obstruct is determined by the Thurston-Bennequin invariant and rotation number seems, at first, a little surprising until one carefully compares the  Pontryagin--Thom construction with a relative version of the Pontryagin--Thom construction. Geiges and independently, Klukas, in private communication, informed the author of another way to deal the the homotopy issues and prove the above theorems. Similar theorems, with extra hypotheses, concerning Legendrian isotopy were proven in \cite{DingGeiges09, Dymara??}.


In \cite{Tchernov03} Chernov defined relative versions of the rotation number for all (not necessarily null-homologous) Legendrian knots as long as the ambient manifold is irreducible and atoroidal or the Euler class of the contact structure is a torsion class (or the contact structure is tight). In~\cite{Chernov05, Tchernov03} he defined relative versions of the self linking invariant (not to be confused with the self-lining number of a transverse knot) for all (not necessarily null-homologous) framed knots in atoroidal and other manifolds. This gives the definition of the relative Thurston-Bennequin number for Legendrian knots in contact manifolds of such topological type. There should be theorems analogous to those above for non-null-homologous knots in this situation, but the precise statements would necessarily be more complicated. It would be interesting to see the precise extension of the above theorems to all non-null-homologous knots in all overtwisted contact manifolds. 
See \cite{DingGeiges09, Dymara??} for partial results along these lines. 

\subsection{Non-loose transverse knots and the coarse classification of transverse fibered knots realizing the Bennequin bound}

In \cite{ColinGirouxHonda09}, Colin, Giroux and Honda proved that if one fixes a knot type in an atoriodal 3--manifold and a tight contact structure on the manifold then there are only finitely many Legendrian knots with given Thurston-Bennequin invariant and rotation number.  
While this does not imply the same finiteness result for transverse knots, it seems likely that such a finiteness result is true for such knots too. 
Surprisingly this finiteness result is far from true in an overtwisted contact structure. We begin with some notation and terminology.
We call the open book for $S^3$ with binding the unknot the \dfn{trivial open book decomposition} and say an open book decomposition is \dfn{non-trivial} if it is not diffeomorphic to the trivial open book decomposition. 

If $(M,\xi)$ is a contact 3--manifold and $\mathcal{K}$ is a topological knot type we denote by $\mathcal{T}(\mathcal{K})$ the set of all transverse knots in the knot type $\mathcal{K}$ up to contactomorphism (co-orientation preserving and smoothly isotopic to the identity). (We note that in some contexts, one might want this to denote the transverse knots up to transverse isotopy, but in this paper we will only consider the coarse classification of knots.) If $n$ is an integer then 
\[
\mathcal{T}_n(\mathcal{K})=\{T\in \mathcal{T}(\mathcal{K}): \sl(T)=n\}
\]
is the set of transverse knots in the knot type $\mathcal{K}$ with self-linking number $n$. 

Given a null-homologous knot $K$ we will denote the maximal Euler characteristic for a Seifert surface for $K$ by $\chi(K)$.

\begin{thm}\label{thm:infinite1}
Let $(B,\pi)$ be a non-trivial open book decomposition with connected binding of a closed 3--manifold $M$ and let $\xi_B$ be the contact structure it supports. 
Denote by $\xi$ the contact structure obtained from $\xi_B$ by a full Lutz twist along $B$.  
Then $\mathcal{T}_{-\chi(B)}(B)$ contains infinitely many distinct non-loose transverse knots up to contactomorphism (and hence isotopy too).
\end{thm}
This theorem gives the first known example of a knot type and contact structure which supports an infinite number of transverse knots with the same self-linking number. 
We have a similar result using half-Lutz twists. 
\begin{thm}\label{thm:infinite2}
Let $(B,\pi)$ be a non-trivial  open book decomposition with connected binding of a closed 3--manifold $M$ and let $\xi_B$ be the contact structure it supports.
Denote by $\xi$ the contact structure obtained from $\xi_B$ by a half Lutz twist along $B$.  
Then $\mathcal{T}_{\chi(B)}(B)$ contains infinitely many distinct non-loose transverse knots up to contactomorphism (and hence isotopy too).
\end{thm}

Finally we have the following less specific but more general theorem along these lines.
\begin{thm}\label{GeneralInfinite}
Let $K$ be a null-homologous knot type in a closed irreducible 3-manifold $M$ with Seifert genus $g>0$. There is an overtwisted contact structure for which $\mathcal{T}_{2g-1}(K)$ is infinite and a distinct overtwisted contact structure for which $\mathcal{T}_{-2g+1}(K)$ is infinite. 
\end{thm}

\begin{remark}
We note that as mentioned above, and proven in Corollary~\ref{nonlunknot} below, any transverse unknot in an overtwisted contact manifold is loose. Hence the previous theorem implies that the unknot is the unique null-homologous knot in any (irreducible) manifold that does not have non-loose transverse representatives in some overtwisted contact manifold (in fact all other null-homologous knots have non-loose transverse representatives in at least two overtwisted contact manifolds). The reason for this is simply that the unknot is the unique knot whose complement has compressible boundary. 

Since, as noted above, Legendrian approximations of non-loose transverse knots are also non-loose, we see that all null-homologous knots have non-loose Legendrian representatives in some overtwisted contact structure. \hfill \qed 
\end{remark}

Using results from \cite{EtnyreVanHornMorris10} we can refine Theorem~\ref{thm:infinite1} for hyperbolic knots to give the coarse classification of transverse knots in a fibered hyperbolic knot type that realize the upper bound in Proposition~\ref{prop:nlsl}.
\begin{thm}\label{thm:tclassify}
Let $(B,\pi)$ be an open book decomposition with connected binding of a closed 3--manifold $M$ and let $\xi_B$ be the contact structure it supports.
Denote by $\xi$ the contact structure obtained from $\xi_B$ by a full Lutz twist along $B$.
If $B$ is a hyperbolic knot then 
\[
\mathcal{T}_{-\chi(B)}(B)=\{K_*\}\cup\{K_i\}_{i\in A},
\]
where $A=\N$ if $\xi_B$ is tight and $A=\N\cup \{0\}$ if not; and if $\xi'$ is any overtwisted contact structure not isotopic to $\xi$ then
\[
\mathcal{T}_{-\chi(B)}(B)=\{K_*\}.
\]
 Moreover 
\begin{itemize}
\item $K_*$ is loose and the knot $K_i$ is non-loose and has Giroux torsion $i$ along a torus parallel to the boundary of a neighborhood of $K_i$.
\item The Heegaard-Floer invariants of all knots in $\mathcal{T}_{-\chi(B)}(B)$ vanish except for $K_0$, which only exists if $\xi_B$ is overtwisted.
\item All knots in $\mathcal{T}_{-\chi(B)}(B)$, except possibly $K_0$, if it exists, become loose after a single stabilization. 
\end{itemize}
\end{thm}

We note that $K_0$ in this theorem might or might not be loose after a single stabilization, but it will certainly become loose after a sufficiently large number of stabilizations. 

\begin{remark}
We note that while Theorem~\ref{thm:infinite2} does allow us to conclude that for the binding of an open book there is at least one overtwisted contact structure such that $\mathcal{T}_{\chi(B)}(B)$ is infinite, the technology in \cite{EtnyreVanHornMorris10} does not appear strong enough to prove a results similar to Theorem~\ref{thm:tclassify} in this case. \hfill \qed 
\end{remark}

\subsection{Non-loose Legendrian knots and the coarse classification of Legendrian knots}
We can use the coarse classification of transverse knots to understand some Legendrian knots in the same knot types. If $(M,\xi)$ is a contact 3--manifold and $\mathcal{K}$ is a topological knot type, then we denote by $\mathcal{L}(\mathcal{K})$ the set of all Legendrian knots in the knot type $\mathcal{K}$ up to contactomorphism (co-orientation preserving and smoothly isotopic to the identity). If $m$ and $n$ are two integers then 
\[
\mathcal{L}_{m,n}(\mathcal{K})=\{L\in \mathcal{L}(\mathcal{K}): \rot(L)=m  \text{ and }  \tb(L)=n\}
\]
is the set of Legendrian knots in the knot type $\mathcal{K}$ with rotation number $m$ and Thurston-Bennequin invariant equal to $n$.

\begin{thm}\label{thm:lclassify}
Let $(B,\pi)$ be an open book decomposition with connected binding of a closed 3--manifold $M$ and let $\xi_B$ be the contact structure it supports.
Denote by $\xi$ the contact structure obtained from $\xi_B$ by a full Lutz twist along $B$.
If $B$ is a hyperbolic knot then there is an $m\in \Z\cup\{\infty\}$ depending only on $B$ such that for each fixed integer $n$ we have
\[
\mathcal{L}_{ \chi(B)+n,n}(B)=\{L_*\}\cup\{L_{n,i}\}_{i\in A_n},
\]
where $A_n=\N$ if $\xi_B$ is tight or $n>m$ and $A_n=\N\cup \{0\}$ if not; and if $\xi'$ is any overtwisted contact structure not isotopic to $\xi$ then
\[
\mathcal{L}_{ \chi(B)+n,n}(B)=\{L_*\}.
\]
See Figure~\ref{infg}. Moreover 
\begin{itemize}
\item $L_*$ is loose and the knot $L_{n,i}$ is non-loose and has Giroux torsion $i$ along a torus parallel to the boundary of a neighborhood of $L_{n,i}$.
\item The Heegaard-Floer invariants of all knots in $\mathcal{L}_{\chi(B)+n,n}(B)$ vanish except for $L_{n,0}$, which is non-zero. (Recall $L_{n,0}$ only exists if $\xi_B$ is overtwisted and $n\leq m$.) 
\item All knots in $\mathcal{L}_{\chi(B)+n,n}(B)$, except possibly $L_{n,0}$, if it exists, become loose after a single positive stabilization.  
\item The negative stabilization of $L_{n,i}$ is $L_{n-1,i}$.
\end{itemize}
\end{thm}
\begin{figure}[ht] 
  \relabelbox{
  \centerline{\includegraphics[width=2.8in]{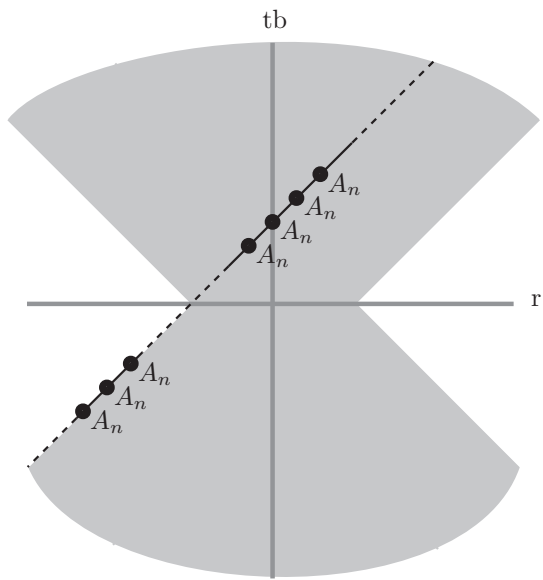}}}
  \relabel{1}{$A_n$}
  \relabel{2}{$A_n $}
  \relabel{3}{$A_n $}
  \relabel{4}{$A_n $}
  \relabel{5}{$A_n $}
  \relabel{6}{$A_n $}
  \relabel{7}{$A_n $}
  \relabel{t}{$\tb$}
  \relabel{r}{$\rot$}
  \endrelabelbox
  \caption{The knots discussed in Theorem~\ref{thm:lclassify}. A dot at $(m,n)$ indicates there is a non-loose knot with rotation number $m$ and Thurston-Bennequin invariant $n$. The subscript on the dot indicates the set that indexes the family of non-loose Legendrian knots with invariants given by the coordinates. The line is given by $-m+n=-\chi(B)$.}
  \label{infg}
\end{figure}   
As we noted in the transverse case, this theorem provides the first examples of a knot type and contact structure that support an infinite number of Legendrian knots with the same classical invariants. We also remark that if the hypothesis on $B$ being a hyperbolic knot is dropped then, as in Theorem~\ref{thm:infinite1} for transverse knots, one can still conclude that there are infinitely many Legendrian knots with fixed invariants as in the theorem, but we cannot necessarily give a classification of these Legendrian knots. 

We notice now that if $-B$ and $B$ are smoothly isotopic then we can extend the above classification.
\begin{thm}\label{thm:linvertableclassify}
Given the hypothesis and notation of Theorem~\ref{thm:lclassify}, suppose that $-B$ is smoothly isotopic to $B$,  where $-B$ denotes $B$ with its orientation reversed. Then for each fixed integer $n$ we have 
\[
\mathcal{L}_{-\chi(B)-n,n}(B)=\{L_*\}\cup\{L_{n,i}\}_{i\in A_n},
\]
where $A_n=\N$ if $\xi_B$ is tight or $n>m$ and $A_n=\N\cup \{0\}$ if not; and if $\xi'$ is any overtwisted contact structure not isotopic to $\xi$ then
\[
\mathcal{L}_{-\chi(B)-n,n}(B)=\{L_*\}.
\]
Reflecting Figure~\ref{infg} about the $\rot=0$ line depicts the knots described here.  Moreover 
\begin{itemize}
\item $L_*$ is loose and the knot $L_{n,i}$ is non-loose and has Giroux torsion $i$ along a torus parallel to the boundary of a neighborhood of $L_{n,i}$.
\item The Heegaard-Floer invariants of all knots in $\mathcal{L}_{-\chi(B)-n,n}(B)$ vanish except for $L_{n,0}$, which is non-zero. (Recall $L_{n,0}$ only exists if $\xi_B$ is overtwisted and $n\leq m$.) 
\item All knots in $\mathcal{L}_{-\chi(B)-n,n}(B)$, except possibly $L_{n,0}$ if it exists, become loose after a single negative stabilization.  
\item The positive stabilization of $L_{n,i}$ is $L_{n+1,i}$.
\end{itemize}
\end{thm}

\begin{remark}
There is an amusing corollary of Theorems~\ref{thm:lclassify} and~\ref{thm:linvertableclassify}. Consider the set $\widetilde{\mathcal{L}}_{r,t}(B)$, where the tilde indicates that we are considering Legendrian knots up to isotopy not just contactomorphism (co-orientation preserving and isotopic to the identity). There is a natural surjective map $\widetilde{\mathcal{L}}_{r,t}(B)\to {\mathcal{L}}_{r,t}(B)$. While it is suspected that this map is not injective in general (though some times it is, for example when the ambient contact manifold is $S^3$ with its standard contact structure), it is difficult to come by specific non-injective examples. In the following paragraph we show how to use Theorems~\ref{thm:lclassify} and~\ref{thm:linvertableclassify} to construct the first such examples; that is, examples where one knows there is a difference between the classification Legendrian knots up to isotopy and the classification up to contactomorphism (isotopic to the identity). We also note that this fact proves that the space of contact structures on $M$ has a non-trivial loop based at the contact structure considered in the theorem.

Under the hypothesis of Theorem~\ref{thm:linvertableclassify}, given the non-loose Legendrian $L_{-\chi,i}\in \mathcal{L}_{0,-\chi(B)}(B)$ from Theorem~\ref{thm:lclassify} and the non-loose knot $L'_{-\chi,i}\in \mathcal{L}_{0,-\chi(B)}(B)$ from Theorem~\ref{thm:linvertableclassify} we know $L_{-\chi,i}$ and $L_{-\chi,i}'$ are contactomorphic by Theorem~\ref{thm:lclassify}. But notice that they cannot be Legendrian isotopic since a positive stabilization of $L_{-\chi,i}$ will result in a loose knot while a positive stabilization of $L_{-\chi,i}'$ will yield a non-loose knot. (Recall stabilization is a well-defined operation and hence Legendrian isotopic knots cannot have different stabilizations.) Notice that this fact concerning stabilizations does not prohibit $L_{-\chi,i}$ and $L_{-\chi,i}'$ from being contactomorphic since a contactomorphism can reverse the sense of stabilization. 
(To clarify this last statement, we notice from the proof of Theorem~\ref{thm:linvertableclassify} that $L_{-\chi,i}'=-L_{-\chi,i}$. Now let $N$ be a standard neighborhood of a Legendrian knot $L$ and let $N'\subset N$ be a standard neighborhood of a stabilization of $L$. The region $N\setminus N'=T^2\times [0,1]$ is a basic slice, see \cite{Honda00a} more on this terminology, and the sign of the basic slice is determined by the sign of the stabilization. Moreover, the sign of the stabilization is also determined by the orientation on $L$. Thus when considering a contactomorphism from $L_{-\chi,i}$ to $-L_{-\chi,i}=L_{-\chi,i}'$ we get an induced contactomorphism of their complements. Stabilizing one will add a basic slice to its complement and since the contactomorphism under consideration reverses the orientation on the underlying knot it will reverse the sign of the basic slice and hence the sense of the stabilization. We notice that the contactomorphism of the complements of $L'_{-\chi,i}$ and $L_{-\chi, i}$ extends to a co-orientation reversion contactomorphism of the complements of their stabilizations. We also notice that this phenomena can only happen since, as plane fields, $\xi$ and $-\xi$ are isotopic and the rotation numbers of the knots in question are zero.)\hfill\qed
\end{remark}

Recall that two Legendrian knots with the same classical invariants will eventually become isotopic after sufficiently many positive and negative stabilizations. Trying to understand exactly how many is an interesting question. We have the following partial results.

\begin{prop}\label{prop:stabilize}
Let $(B,\pi)$ be an open book decomposition with connected binding of a closed 3--manifold $M$ and let $\xi_B$ be the contact structure it supports.
Assume that $\xi_B$ is tight and $-B$ is not isotopic to $B$. (Here $-B$ denotes $B$ with the reversed orientation.) 
Consider the contact structure $\xi$ obtained from $\xi_B$ by a full Lutz twist along $B$.
If a knot $L\in \mathcal{L}(B)$ satisfies
\[
\tb(L)- \rot(L)> -\chi(B)
\]
then it becomes loose after $\frac 12(-\chi(B)-\tb(L))$ or fewer positive stabilizations. 
\end{prop}
There is a similar statement for the case when $B$ is isotopic to $-B$. 
In the next subsection we discuss the ``geography problem" for Legendrian knots and in particular at the end of the section we discuss the significance of this proposition. 

\begin{remark}
We notice that examining the proof of Propostion~\ref{prop:stabilize} allows one to drop the hypothesis that $-B$ is not isotopic to $B$ if $\tb(L)+\rot(L)\not=-\chi(B).$\hfill\qed
\end{remark}

\subsection{The geography for non-loose knots}
Let $\mathcal{K}$ be a knot type on a contact manifold $(M,\xi)$. We have the map
\[
\Phi:\mathcal{L}(\mathcal{K})\to \Z\times \Z: L \mapsto (\rot(L), \tb(L)).
\]
Determining which pairs of integers are in the image of $\Phi$ is called the \dfn{Legendrian geography problem}. 
The image if $\Phi$ is frequently called the \dfn{Legendrian mountain range} of $\mathcal{K}$ because in the case that $\xi$ is the tight contact structure on $S^3$ the image resembles the silhouette of a mountain range.  This structure comes from the facts that when $\xi$ is the tight contact structure on $S^3$ we know that (1) the image is symmetric, (2) the Thurston-Bennequin invariant is bounded above, and (3) positive and negative stabilizations show that the pairs $(m-k,n+l)$, where $k\geq 0, |l|\leq k$, and $l+k$ is even, are in the image if $(m,n)$ is. (Recall that for any Legendrian knot $L$ we must have that $\tb(L)+\rot(L)$ is odd, so (3) says all possible pairs of points in the cone with with sides of slope $\pm 1$ and top vertex $(m,n)$ are realized by Legendrian knots that are stabilizations of the given knot.)  For a general contact structure we only have (3) and if the structure is tight (2). 

Once one understand the Legendrian geography problem for a knot type $\mathcal{K}$ then the classification of Legendrian knots will be complete when one understands the preimage of each pair of integers. Determining $\Phi^{-1}(m,n)$ is known as the \dfn{botany problem for $\mathcal{K}$.} The \dfn{weak botany problem} asked to determine if $\Phi^{-1}(m,n)$ is empty, finite or infinite for each pair $(m,n)$. As mentioned above if $\mathcal{K}$ is a hyperbolic knot and $\xi$ is tight then the preimage of any $(m,n)$ is either empty or finite. 

If $\xi$ is overtwisted then every pair of integers $(m,n)$ for which $m+n$ is odd is in the image of $\Phi$. Thus the general geography problem is not interesting for overtwisted contact structures and so we need to restrict our attention to non-loose knots. Let $\mathcal{L}^{nl}(\mathcal{K})$ be the set of non-loose Legendrian knots in the knot type $\mathcal{K}$. Determining $\Phi(\mathcal{L}^{nl}(\mathcal{K}))$ will be called the Legendrian geography problem when $\xi$ is overtwisted, which is what we consider from this point on. Similarly when we discuss the weak botany problem for overtwisted contact manifolds we will only be considering the preimage of $(m,n)$ that lie in $\Phi(\mathcal{L}^{nl}(\mathcal{K}))$ (since by Theorem~\ref{thm:legloose} we know the preimage contains exactly 1 loose Legendrian knot if $n+m$ is odd and none if it is even).

We now summarize what we know about the geography problem for non-loose knots. First, Proposition~\ref{looseLegbound} implies that (assuming the knot under consideration is not the unknot!) the image of $\Phi$ (when restricted to non-loose knots) is contained in the region shown in Figure~\ref{nonlooseregion}. The four black lines $l_1,\ldots, l_4$, are the lines 
\[
\pm\tb(L)\pm \rot(L)= -\chi(\Sigma)
\]
(where all combinations of $\pm$ are considered). The region is broken into 7 subregions, $R_1,\ldots, R_7$ as indicated in the figure, by the four lines. The regions $R_i$ are open regions, when discussing the corresponding closed regions we will of course us the notations $\overline R_i$.

\begin{figure}[ht] 
  \relabelbox{
  \centerline{\epsfbox{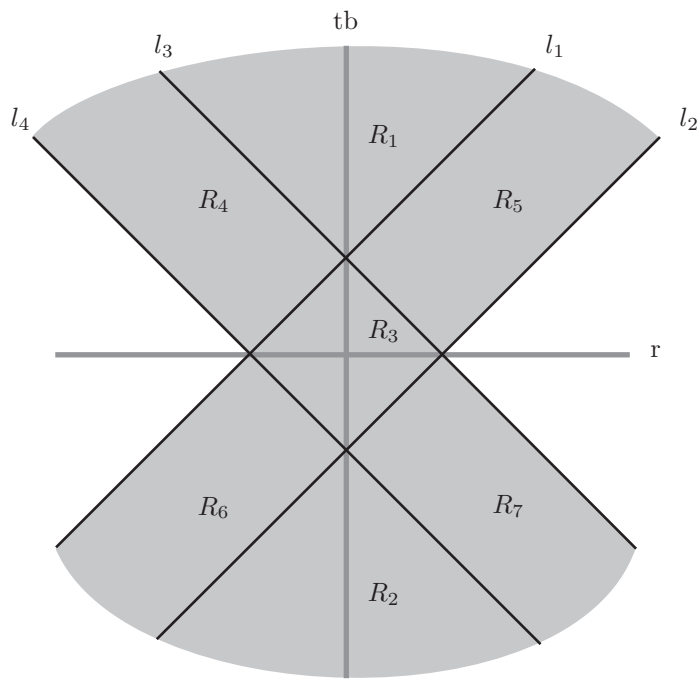}}}
  \relabel{1}{$R_1$}
  \relabel{2}{$R_2$}
  \relabel{3}{$R_3$}
  \relabel{4}{$R_4$}
  \relabel{5}{$R_5$}
  \relabel{6}{$R_6$}
  \relabel{7}{$R_7$}
  \relabel{t}{$\tb$}
  \relabel{r}{$\rot$}
  \relabel{a}{$l_1$}
  \relabel{b}{$l_2$}
  \relabel{c}{$l_3$}
  \relabel{d}{$l_4$}
  \endrelabelbox
  \caption{A non-loose knot must map to the grey region under the map $\Phi$. The black lines, labeled $l_1,\ldots, l_4$, break the grey region into 7 subregions as indicated in the figure. The regions $R_i$ are open regions (that is not including the points on the black lines).}
  \label{nonlooseregion}
\end{figure}   
\begin{quest}
Can there be a non-loose knot with image in $R_1$ or $R_2$?
\end{quest}
While we believe the answer to this question is likely ``yes'', it seems fairly likely that the answer to the next, related, questions is ``no''.
\begin{quest}
Can there be a $(m,n)\in R_1\cup R_2$ with $\Phi^{-1}((m,n))$ infinite?
\end{quest}
We know from Colin, Giroux and Honda \cite{ColinGirouxHonda09} that if you fix a tight contact structure on a manifold, a knot type and two integers, there are finitely many Legendrian knots in the given knot type realizing the integers as their Thurston-Bennequin invariant and rotation number. This seems close to proving the answer to the following question is ``yes''.
\begin{quest}\label{quest3}
Given a knot type $\mathcal{K}$ in a manifold $M$ are there only finitely many overtwisted contact structures on $M$ such that $\mathcal{K}$ can have non-loose representatives?
\end{quest}
In fact the answer to this question would be ``yes'' if the following general question about Legendrian knots could be answered in the affirmative.
\begin{quest}
Given a knot type $\mathcal{K}$ in $M$ is there a number $n$ such that if $L$ is any Legendrian representative in any contact structure on $M$ with $\tb(L)<n$ then $L$ destabilizes?
\end{quest}
Of course this question is also very interesting if one first fixes the contact structure and then asks for the integer $n$. If this more restricted question had a positive answer then the Colin-Giroux-Honda result would imply that for any fixed tight contact structure $\xi$ on a manifold $M$ the Legendrian knots in a given knot type $\mathcal{K}$ would be ``finitely generated'', by which we mean there would be a finite number of non-destabilizable Legendrian knots in $\mathcal{L}(\mathcal{K})$ such that all other elements in  $\mathcal{L}(\mathcal{K})$ would be stabilizations of these. This would also be true in overtwisted contact structures up to Lutz twisting along the knot.

We think it is very likely the answer to the following weaker version of Question~\ref{quest3} is ``yes".
\begin{quest}
Given a knot type $K$ in a manifold $M$ are there only finitely many overtwisted contact structures on $M$ such that $\Phi$ can have infinite preimage at some point? (Maybe even at most two such structures.)
\end{quest}

With the notation established for Figure~\ref{nonlooseregion} we discuss Proposition~\ref{prop:stabilize}. The proposition says that if $L$ is a Legendrian representative with invariants in the region $\overline{R_1\cup R_4}$, then once it is positively stabilized into region $R_5$ or $R_3$ it will be loose. This is surprising given that according to the bounds given in Proposition~\ref{looseLegbound} a knot with invariants in $R_4$ could theoretically be stabilized positively an arbitrary number of times and stay non-loose.

\noindent
{\bf Acknowledgments:} The author thanks Amey Kolati, Lenny Ng, and Bulent Tosun for useful discussions and e-mail exchanges during the preparation of this work. He also thanks Thomas Vogel for agreeing to allow the inclusion of Theorem~\ref{mainS3contacto} that he and the author worked out some years ago. The author also thanks Hansj\"org Geiges and Mirko Klukas for interesting discussions concerning the classification of loose Legendrian knots. The author is also grateful to the referees of this paper who made valuable comments that helped clarify many points in the paper.  This work was partially supported by NSF grant DMS-0804820.

\section{Background concepts}
We assume the reader is familiar with convex surface theory and Legendrian knots, see for example \cite{Etnyre05, EtnyreHonda01b}. For convenience we recall some of the key features of Legendrian and transverse knots used in this paper in Subsection~\ref{revLegTrans}. In the following subsection we recall Eliashberg and Fraser's classification of non-loose Legendrian unknots. We sketch a simple proof of this result that T.~Vogel and the author had worked out, and observe the immediate corollary concerning non-loose transverse unknots. In Subsetction~\ref{gtsec} we recall the definition of Giroux torsion and make several observations necessary for the proofs of our main results. In the last subsection we recall the notion of quasi-convexity introduced in \cite{EtnyreVanHornMorris10}.

\subsection{Neighborhoods of Legendrian and transverse knots}\label{revLegTrans}
Recall that a convex torus $T$ in a contact manifold $(M,\xi)$ will have an even number of dividing curves of some slope. We call the slope of the dividing curves on a convex torus the \dfn{dividing slope} of the torus. If there are just two dividing curves then using the Legendrian realization principle of Giroux we can arrange that the characteristic foliation has two lines of singularities parallel to the dividing curves, these are called \dfn{Legendrian divides}, and the rest of the foliation is by curves of some slope not equal to the dividing slope. These curves are called \dfn{ruling curves} and their slope is called the \dfn{ruling slope}. Any convex torus with such a characteristic foliation will be said to be in standard form. Note that given a torus in standard form we can perturb the foliation to have two closed leaves parallel to the dividing curves and the other leaves spiraling from one closed leaf to the other.  

The regular neighborhood theorem for Legendrian submanifolds says that given a Legendrian knot $L$ in a contact manifold $(M,\xi)$ there is some neighborhood $N$ of $L$ that is contactomorphic to a neighborhood $N'$ of the image of the $x$-axis in $\R^3/(x\mapsto x+1)\cong S^1\times \R^2$ with contact structure $\xi_{std}=\ker(dz-y\, dx)$. By shrinking $N$ and $N'$ if necessary we can assume that $N'$ is a disk in the $yz$-plane times the image of the $x$-axis. It is easy to see, using the model $N'$, that $\partial N$ is a convex torus with two dividing curves of slope $\frac 1n$ where $n=\tb(L)$. Thus we can assume that $\partial N$ is in standard form. Moreover, notice that $L_\pm=\{(x, \pm\epsilon, 0)\}\subset N'$ is a $(\pm)$-transverse curve. The image of $L_+$ in $N$ is called the \dfn{transverse push-off} of $L$ and $L_-$ is called the \dfn{negative transverse push-off}. One may easily check that $L_\pm$ is well-defined and compute that 
\[
\sl(L_\pm)=\tb(L)\mp \rot(L).
\]

We now recall how to understand stabilizations and destabilizations of a Legendrian knot $K$ in terms of the standard neighborhood. Inside the standard neighborhood $N$ of $L$ we can positively or negatively stabilize $L$. Denote the result $S_+(L)$, respectively $S_-(L)$. Let $N_s$ be a neighborhood of the stabilization of $L$ inside $N$. As above we can assume that $N_s$ has convex boundary in standard form. It will have dividing slope $\frac{1}{n-1}$. Thus the region $N\setminus N_s$ is diffeomorphic to $T^2\times[0,1]$ and the contact structure on it is easily seen to be a ``basic slice'', see \cite{Honda00a}. There are exactly two basic slices with given dividing curves on their boundary and as there are two types of stabilization of $L$ we see that the basic slice $N\setminus N_s$ is determined by the type of stabilization done, and vice versa. Moreover if $N$ is a standard neighborhood of $L$ then $L$ destabilizes if the solid torus $N$ can be thickened to a solid torus $N_d$ with convex boundary in standard form with dividing slope $\frac 1{n+1}$. Moreover the sign of the destabilization will be determined by the basic slice $N_d\setminus N$.

Denote by $S_a$ the solid torus $\{(\phi, (r,\theta))| r\leq a\}\subset S^1\times \R^2$, where $(r,\theta)$ are polar coordinates on $\R^2$ and $\phi$ is the angular coordinate on $S^1$, with the contact structure $\xi_{cyl}=\ker (d\phi + r^2\, d\theta)$.
Given a transverse knot $K$ in a contact manifold $(M,\xi)$ one may use a standard Moser type argument to show that there is a neighborhood $N$ of $K$ in $M$ and some positive number $a$ such that $(N,\xi|_\xi)$ is contactomorphic to $S_a$. Notice that the tori $\partial S_b$ inside of $S_a$ have linear characteristic foliations of slope $-b^2$. Thus for all integers with $\frac 1{\sqrt n}<a$ we have tori $T_n=\partial S_{1/\sqrt{n}}$ with linear characteristic foliation of slope $-\frac 1n$. Let $L_n$ be a leaf of the characteristic foliation of $T_n$. Clearly $L_n$ is a Legendrian curve in the same knot type as $T$ and $\tb(L_n)=-n$. Any Legendrian $L$ Legendrian isotopic to one of the $L_n$ so constructed will be called a \dfn{Legendrian approximation} of $K$. 

\begin{lem}[Etnyre-Honda 2001, \cite{EtnyreHonda01b}]\label{legapprox}
If $L_n$ is a Legendrian approximation of the transverse knot $K$ then $(L_n)_+$ is transversely isotopic to $K$. Moreover, $L_{n+1}$ is Legendrian isotopic to the negative stabilization of $L_n$.\qed
\end{lem}

\subsection{Non-loose knots in $S^3$}\label{nlunknotclass}
In this section we recall the following coarse classification of unknots in overtwisted contact structures on $S^3$. To state the theorem we recall that the homotopy class of a plane field $\xi$ on $M$ can be determined by computing two invariants \cite{Gompf98}: the 2-dimensional invariant, which is determined by the $\text{spin}^c$ structure $\mathfrak{s}_\xi$ associated to the plane field (if $H^2(M;\Z)$ has no 2-torsion then $\mathfrak{s}_\xi$  is determined by the Euler class $e(\xi)$ of $\xi$) and the 3-dimensional invariant $d_3(\xi)$. In particular, on the 3--sphere the 2-dimensional invariant of a plane field vanishes and so it is determined by its 3-dimensional invariant $d_3(\xi) \in \Z$. 

\begin{thm}[Eliashberg-Fraser 2009, \cite{EliashbergFraser09}]\label{mainS3contacto}
There is a unique overtwisted contact structure $\xi$ on $S^3$ that contains non-loose Legendrian unknots. This contact structure has $d_3(\xi)=1$ and is shown in Figure~\ref{fig:uknot}. Every Legendrian unknot with tight 
complement has $\tb>0$ and up to contactomorphism there are exactly two such knots with 
$\tb=n>1$, they are distinguished by their rotation numbers which are $\pm(n-1)$, and a unique 
such knot with $\tb=1$.
\end{thm}
\begin{figure}[ht] 
  \relabelbox{
  \centerline{\epsfbox{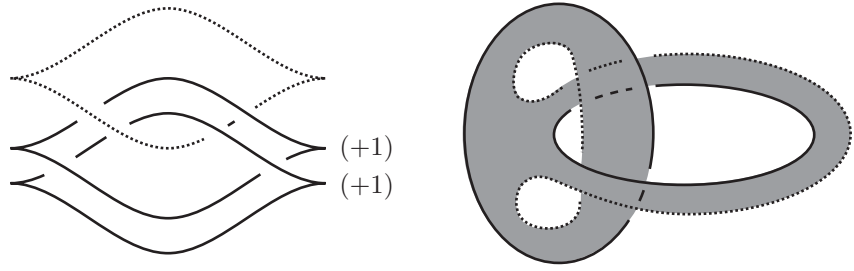}}}
  \relabel{1}{$(+1)$}
  \relabel{2}{$(+1)$}
  \endrelabelbox
  \caption{The two contact surgeries on the left give the contact structure $\xi$ in 
  Theorem~\ref{mainS3contacto}. The dotted Legendrian becomes a Legendrian unknot $L$ with $tb=1$
  in the surgered manifold. The disk $L$ bounds is indicated on the right.}
  \label{fig:uknot}
\end{figure}    
We include a brief proof of this result that the author worked out with Thomas Vogel but never published. This is essentially the same as the proof in \cite{EliashbergFraser09}, though it is couched in somewhat different language and we identify a surgery picture for the minimal $\tb$ example in Figure~\ref{fig:uknot}. Before giving the proof we notice two simple corollaries concerning transverse unknots and non-loose unknots in other contact manifolds. 
\begin{cor}\label{nonlunknot}
Any transverse unknot in any overtwisted contact manifold is loose and hence coarsely determined by its self-linking number.
\end{cor}
\begin{proof}
Since any non-loose Legendrian unknot can be negatively stabilized until it is loose, Lemma~\ref{legapprox} implies that any transverse unknot is the transverse push-off of a loose Legendrian. Since the transverse push-off can be done in any $C^\infty$-neighborhood of the Legendrian knot it is clear that the overtwisted disk in the complement of the Legendrian is also in the complement of the transverse knot. Thus any transverse unknot in an overtwisted contact structure is loose. The rest follows from Theorem~\ref{thm:transverseloose}
\end{proof}

\begin{cor}
On a fixed 3-manifold $M$ with fixed $\text{spin}^c$ structure $\mathfrak{s}$
there is an overtwisted contact structure having non-loose 
unknots and associated $\text{spin}^c$ structure $\mathfrak{s}$
if and only if there is a tight contact structure on $M$ with associated $\text{spin}^c$ structure $\mathfrak{s}$. 
Moreover, the number of such overtwisted contact structure associated to $\mathfrak{s}$ is equal to the number of tight contact structures associated to $\mathfrak{s}$.
\end{cor}
\begin{proof}
Let $U$ be an unknot in $M$ and $U'$ be an unknot in $S^3$.
The unique prime decomposition of tight contact manifolds \cite{Colin97} implies that any tight contact structure on
$M\setminus U$ comes from tight contact structures on $M$ and $S^3\setminus U'$. The result follows. 
\end{proof}
We now turn to the classification of non-loose unknots in $S^3$.
\begin{proof}[Proof of Theorem~\ref{mainS3contacto}]
Let $\xi$ be any overtwisted contact structure on $S^3$ and $L$ a Legendrian unknot in $(S^3,\xi)$.
We can decompose $S^3$ as 
\[S^3=V_1\cup_\phi V_2\]
where $V_i=S^1\times D^2$ and $\phi:\partial V_1\to \partial V_2$ is the map given by 
\[\phi=\begin{pmatrix} 0&1 \\ 1&0\end{pmatrix}\]
in meridian-longitude coordinates on each $\partial V_i$. Moreover we can assume that 
$V_1$ is a standard neighborhood of $L$ with convex boundary having dividing slope 
$1/n$ where $n=tb(L)$.
If we assume the complement of $L$ is tight then the contact structure on $V_2$ is tight (of course,
the contact structure on $V_1$ is always tight). The boundary of $V_2$ is convex and has dividing slope
$n$. Clearly $L$, up to contactomorphism, is determined by the contact structure on $V_2$. In \cite{Giroux00, Honda00a} tight contact structures on solid tori were classified and this leads to the following possibilities for tight contact structures on $V_2$. 
If $n<0$ then there are precisely $|n|$ distinct tight contact structures on $V_2$ with the given dividing
curves on the boundary. All of the contact structures are realized as the complement of a $tb=n$,
Legendrian unknot in the standard tight contact structure on $S^3$. Thus they never show up 
as the complement of an unknot in an overtwisted contact structure.  If $n=0$ then $V_2$ has
dividing slope 0 and hence the contact structure must be overtwisted. Thus $n\not=0$ if the complement
of $L$ is tight. Finally, if $n>1$ then there are two tight contact structures on $V_2$ and only one tight
contact structure on $V_2$ when $n=1$. Thus we have shown there is a unique Legendrian knot
in an overtwisted contact structure on $S^3$ that has $tb=1$ and there are at most two when $tb>1$.
We are left to show that these all occur in the same contact structure, that contact structure has
$d_3 = 1$ and all these knots are distinct. We begin by inductively showing that the Legendrian knots
with $tb=n$ are in the same overtwisted contact structure as the knot with $tb=1$. 
To this end let $L$ be the Legendrian unknot with $tb=1$. The unique
tight contact structure on $V_2$ has relative Euler class $e=0$. Thus from \cite{EtnyreHonda01b} 
we see that
$r(L)=0$. Now $V_2$
can be written as $V_2=S\cup N$ where $S$ is as solid torus and $N \cong T^2\times [0,1]$ where
$T^2\times \{1\}=\partial V_2$ and $T^2\times\{0\}$ is glued to $\partial S$. We can arrange that
$\partial S$ is convex with two dividing curves of slope $2$. The contact structure on $N$ is a
basic slice. There are two basic slices and they are distinguished by the sign of the bypass on a
meridional annulus. Moreover, we can realize both signs for $N$ inside $V_2$ as may be easily checked
in a model for $V_2$. Thinking of $N$ as glued to $V_1$
now, instead of $V_2$ we see that $V_1\cup N$ is a solid torus with convex boundary having 
two dividing curves of slope $\frac {1}{2}$. This is a standard neighborhood of a Legendrian 
unknot $L'$ with $tb=2$. Moreover $L$ is a stabilization of $L'$
and which stabilization it is depends on the sign of the basic slice $N$, see Section~\ref{revLegTrans} above.
Thus if $L$ is in the contact 
structure $\xi$ then there are two Legendrian unknots $L_+$ and $L_-$ in $\xi$ such that 
$S_\mp(L_\pm)=L$. So $tb(L_\pm)=2$ and $r(L_+)=1=-r(L_-)$ and we see that $L_+$ and $L_-$ are
distinct Legendrian knots and they exist in the same contact structure as $L$. 

Now suppose we have shown that the Legendrian unknots with $tb\leq n$ and tight complements all
exist in the same overtwisted contact structure and satisfy $r(L)=\pm (tb(L)-1)$. Consider
a Legendrian unknot $L$ with $tb=n$. We can decompose $V_2$ as above, but with $\partial S$
having dividing slope $n+1$. Now $N$ is again a basic slice, but only one basic slice can be a
subset of $V_2$. Using the relative Euler class discussed in 
\cite{EtnyreHonda01b} one may easily see the sign of the bypass determining
the contact structure on $N$ agrees with the sign of $r(L)$. For the rest of the argument we assume the
sign is positive.
As above, $V_1\cup N$ is a neighborhood of a Legendrian unknot $L'$ with
$tb(L')=n+1$ and $S_-(L')=L$. (Note the sign of the bypass switches since we turned $N$ upside down.)
Thus $r(L)=r(L')-1$ and $r(L')=r(L)+1=(n-1)+1=n$.

We are left to identify the contact structure containing the unknot with $tb=1$. This is shown in 
Figure~\ref{fig:uknot}. One may easily compute
\[d_3(\xi)=\frac 14 (c^2(X) - 3\sigma(X) - 2\chi(X))+2=1,\]
where $X$ is the bordism from $S^3$ to $S^3$ given in the figure, $\sigma(X)=0$ its the signature of $X$, $\chi(X)=2$ its Euler
characteristic of $X$ and $c^2=0$ is the square of the ``Chern class'' of the (singular) almost complex structure on
$X$. (For a discussion of this formula see \cite{DingGeigesStipsicz04}.) 
If $L$ is the knot indicated in the figure, then 
Legendrian surgery on $L$ cancels one of the $+1$-contact surgeries in the figure. Thus the 
resulting manifold is the contact manifold obtained by $+1$-contact surgery on the $tb=-1$ unknot
in the standard tight contact structure on $S^3$. This is well known to be the tight contact structure on
$S^1\times S^2$ (see \cite{DingGeigesStipsicz04}). Thus the complement of $L$ (which is a subset of the tight contact
structure on $S^1\times S^2$) is clearly tight. 
Moreover, in Figure~\ref{fig:uknot} we see a disk $L$ bounds in
the surgered manifold, indicating that $L$ is an unknot. This disk gives $L$ a framing that is 2
less than the framing given by a disk in $S^3$ before the surgeries. Thus we see that the 
contact framing with respect to this disk is $1=-1+2$.
\end{proof}

\subsection{Giroux torsion}\label{gtsec}

Given a contact manifold $(M,\xi)$ and an isotopy class of tori $[T]$ in $M$ we define the \dfn{Giroux torsion of $(M,\xi)$ in the isotopy class $[T]$}, denoted $\tor((M,\xi),[T])$,  to be the maximum natural number $n$ such that there exists a contact embedding $\phi : (T^2 \times [0,1], \zeta_{n}) \to (M,\xi)$, where $\phi(T^2\times\{0\})\in [T]$ and the contact structure $\zeta_{n}$ on $T^2 \times [0,1]$ (thought of as $\mathbb{R}^2/\mathbb{Z}^2 \times [0,1]$) is given by $\zeta_{n} = \ker(\sin(2n\pi t)\, dx + \cos(2n\pi t)\, dy)$.  We sometimes refer to ``half--Giroux'' torsion when $T^2\times[0,1/2]$, with the above contact structure, can be contact embedded in $(M,\xi)$.

The \dfn{Giroux torsion of $(M,\xi)$} is the maximal Giroux torsion taken over all isotopy classes of tori in $M$. Recall that the Giroux torsion of $M$ is infinite in all isotopy class of tori if $\xi$ is overtwisted. It is unknown if the Giroux torsion is finite when $\xi$ is tight, though this is frequently the case \cite{Colin01a}. We make a simple observation about Giroux torsion that we will need below.

\begin{lem}\label{lem:torsion2}
Let $(M,\xi)$ be a contact manifold with boundary a torus $T$ and $[T]$ the isotopy class of $T$. Assume the characteristic foliation on $\partial M$ is linear. Then 
\[
\tor ((M,\xi),[T])=\tor((\interior M, \xi|_{\interior M}),[T]),
\]
where $\interior M$ denotes the interior of $M$.
\end{lem}
\begin{proof}
The inequality $\tor ((M,\xi),[T])\geq \tor((\interior M, \xi|_{\interior M}),[T])$ is obvious. For the other inequality assume we have the contact embedding $\phi : (T^2 \times [0,1], \zeta_{n}) \to (M,\xi)$. One of the boundary components of $T^2\times[0,1]$ is mapped into the interior of $M$ (recall $M$ only has one boundary component and it is in the isotopy class $[T]$). Assume that it is $T^2\times\{1\}$ that maps into the interior (the other case being similar). We can extend $\phi$ to an embedding of $T^2\times [0,1+\epsilon]$ for some $\epsilon$, where the contact structure on $T^2\times [0,1+\epsilon]$ is given by $\ker(\cos(2n\pi t) dx + \sin(2n\pi t) dy)$. From this we easily find a contact embedding of $(T^2 \times [0,1], \zeta_{n})$ into the interior of $M$.
\end{proof}

\subsection{Quasi-compatibility}\label{qcsection}

In \cite{EtnyreVanHornMorris10} the notion of quasi-compatibility was introduced.  Let $\xi$ be an oriented contact structure on a closed, oriented manifold $M$ and $(L, \Sigma)$ an open book for $M$.  We say $\xi$ and $(L, \Sigma)$ are {\em quasi-compatible} if there exists a contact vector field for $\xi$ which is everywhere positively transverse to the fibers of the fibration $(M\setminus L)\to S^1$ and positively tangent to $L$.

One can construct contact structure quasi-compatible with an open book using a slight modification of the standard construction of compatible contact structures. Specifically, given the open book $(L,\Sigma)$ we notice that $M-L$ is the mapping torus of some diffeomorphism $\phi:\Sigma\to\Sigma$, where $\phi$ is the identity map near $\partial \Sigma$. According to \cite{Giroux91}, given any collection of closed curves $\Gamma$ on $\Sigma$ that divide $\Sigma$ (in the sense of dividing curves for a convex surface) and are disjoint from $\partial \Sigma$, we can construct an $\R$-invariant contact structure $\xi$ on $\Sigma\times \R$ that induce the curves $\Gamma$ as the dividing curves on $\Sigma\times\{t\}$ for any $t\in \R$. If $\phi(\Gamma)$ is isotopic to $\Gamma$ then we can find (after possibly isotoping $\phi$) a negative function $h:\Sigma\to \R$ such that the top and bottom of the region between $\Sigma\times\{0\}$ and the graph of $h$ can be glued via $\phi$ so that $\xi$ restricted to this region induces a contact structure on $M-L$. After slightly altering this contact structure near $L$ we can then extend this contact structure over $L$ in the standard manner. This gives a contact structure $\xi$ that is quasi-compatible with $(L,\Sigma)$ (and induces the given $\phi$ invariant dividing curves on all the pages). 

One of the main technical results of \cite{EtnyreVanHornMorris10} was the following. 
\begin{thm}[Etnyre and Van Horn-Morris 2010, \cite{EtnyreVanHornMorris10}]\label{thm:maxclass}
Let $(B, \Sigma)$ be a fibered transverse link in a contact 3-manifold $(M, \xi)$ and assume that $\xi$ is tight when restricted to $M\setminus B$.  If $sl_\xi{(B,\Sigma)} = -\chi(\Sigma)$, then $\xi$ is quasi-compatible with $(B,\Sigma)$ and either
\begin{enumerate}
\item $\xi$ is supported by $(B, \Sigma)$ or
\item $\xi$ is obtained from $\xi_{(B,\Sigma)}$ by adding Giroux torsion along tori which are incompressible in the complement of $B$.  \qed
\end{enumerate}
\end{thm}

\section{Observations about non-loose knots}
Though proven in \cite{Dymara??} we give a quick proof of Proposition~\ref{looseLegbound} for the convenience of the reader. Though essentially the same as the proof given in \cite{Dymara??} it uses quite different language.
\begin{proof}[Proof of Proposition~\ref{looseLegbound}]
Notice that the inequaitly to be proved is equivalent to 
\begin{align*}
    \tb(L) \pm \rot(L) \leq -\chi (L)  \quad& \text{if $\tb(L)\leq 0$}, \\
    -\tb(L) \pm \rot(L) \leq -\chi(L)  \quad & \text{if } \tb(L)>0.
\end{align*}
To establish this let $N$ be a standard convex neighborhood of $L$ with ruling slope 0 (that is, given by the Seifert framing) and $L'$ a ruling curve on $\partial N$.
Clearly $L'$ is null-homologous in the complement of $L$. The dividing curves on $\partial N$ have slope $\frac 1{\tb(L)}$ thus
twisting of $L'$ with respect to $\partial N$ is $-|\tb(L)|$; however, since the framing on $L'$ induced by the Seifert surface for $L'$ and
by $\partial N$ are the same we also see that $\tb(L')=-|\tb(L)|$. Checking that $\rot(L)=\rot(L')$ one obtains the desired inequalities from the Bennequin inequality applied to $L'$ in $(M-N)$. To see that $\rot(L)=\rot(L')$ we can trivialize the contact planes in the neighborhood $N$ by extending the unit tangent vector to $L$ to a non-zero section of $\xi$. Then any Legendrian longitude on $\partial N$ that is oriented in the same direction as $L$ clearly was winding zero with respect to this trivialization and thus the rotation numbers of this longitude and $L$ will agree. 
\end{proof}

We are now ready to establish Proposition~\ref{nltnll} that explains the relation between a non-loose transverse or Legendrian knot and it Legendrian approximations or, respectively, transverse push-offs. 
\begin{proof}[Proof of Proposition~\ref{nltnll}]
For the first statement let $L$ be a Legendrian knot such that its positive transverse push-off  $L_+$ is transversely isotopic to $T$. If there is an overtwisted disk $D$ in the complement of $L$ then $D$ is in the complement of some small neighborhood of $L$. Since $L_+$ can taken to be in any neighborhood of $L$ we see that $D$ is in the complement of $L_+$. Thus extending the transverse isotopy of $L_+$ to $T$ to a global contact isotopy we can move $D$ to an overtwisted disk in the complement of $T$.

To prove the second statement in the theorem we note, Corollary~\ref{nonlunknot}, that all transverse unknots in an overtwisted contact structure are loose, but by Theorem~\ref{mainS3contacto} we know there are non-loose Legendrian unknots in a particular contact structure on $S^3$. Thus the transverse push-off of a non-loose Legendrian knot does not need to be non-loose. On the other hand, the proof of Theorem~\ref{thm:lclassify} shows that all the non-loose Legendrian knots considered in that theorem, have non-loose transverse push-offs. 
\end{proof}

The proof of Proposition~\ref{prop:nlsl} is a fairly easy consequence of the previous two proofs. 
\begin{proof}[Proof of Proposition~\ref{prop:nlsl}]
Let $K$ be a transverse knot with $\sl(K)>-\chi(\Sigma)$ and let $L$ be a Legendrian approximation of $K$, so $\tb(L)-\rot(L)>-\chi(\Sigma)$. Notice that if $\tb(L)\leq 0$ then it does not satisfy the bound given in Proposition~\ref{looseLegbound}  and thus $L$ is loose and by the proof of  Proposition~\ref{nltnll} we see that $K$ is loose as well. By Lemma~\ref{legapprox} all negative stabilizations of $L$ will be Legendrian approximations of $K$. Since after stabilizing enough times we can assume that $L$ has non-positive Thurston-Bennequin invariant we see that $K$ must be loose. 
\end{proof}

\section{Loose knots}\label{sec:loose}
In this section we explore the homotopy theory of plane fields in the complement of Legendrian knots. More generally, we study homotopy classes of plane fields on manifolds with boundary. We assume the reader is familiar with the Pontryagin--Thom construction on a closed manifold and its implications for classifying homotopy classes of plane fields on a 3--manifold, see for example \cite{DingGeigesStipsicz04, Gompf98}. We end this section by proving Theorems~\ref{thm:legloose} and~\ref{thm:transverseloose} concerning the coarse classification of loose Legendrian and transverse knots. 
\subsection{Homotopy classes of plane fields on manifolds with boundary}

We begin by recalling the Pontryagin--Thom construction in the context of 3--manifolds with boundary. Let $M$ be an oriented 3--manifold with boundary. The space of oriented plane fields on $M$ is denoted $\mathcal{P}(M)$ and if one is given a plane field $\eta$ along $\partial M$ then the set of oriented plane fields that extend $\eta$ to all of $M$ will be denoted $\mathcal{P}(M,\eta)$. Similarly we denote the space of unit vector fields on $M$ by $\mathcal{V}(M)$ and the set of unit vector fields extending a given unit vector field $v$ on the boundary is denoted by $\mathcal{V}(M,v)$. All of these spaces can be topologized as a space of sections of an appropriate bundle. Also notice that $\mathcal{P}(M,\eta)$ and $\mathcal{V}(M,v)$ might be empty depending on $\eta$ and $v$.

Choosing a metric on $M$ we can identify oriented plane fields in $M$ with unit vector fields on $M$ by sending a plane field $\xi$ to the unit vector field $v$ such that $v$ followed by an oriented basis for $\xi$ orients $TM$. Thus a choice of metrics identifies the following spaces
\[
\mathcal{P}(M)\cong \mathcal{V}(M),\quad \text{and} \quad \mathcal{P}(M,\eta)\cong \mathcal{V}(M,v),
\]
where $v$ is the unit vector field along the boundary of $M$ associated to $\eta$ by the metric and orientation.

It is well known that the tangent bundle of $M$ is trivial. Fixing some trivialization, we identify
$TM=M\times \R^3$ and the unit tangent bundle of $M$ with $M\times S^2$. Using these identifications we can identify the space of unit vector fields $\mathcal{V}(M)$ and the space of smooth functions from $M$ to $S^2$, which we denote by $Maps(M,S^2)$. Similarly if the vector field $v$ along $\partial M$ corresponds to the function $f_v$ then $\mathcal{V}(M,v)$ can be identified with the space of smooth functions from $M$ to $S^2$ that agree with $f_v$ on $\partial M$, which we denote by $Maps(M,S^2; f_v)$.

We now assume that $f_v:\partial M\to S^2$ misses the north pole of $S^2$ (and hence is homotopic to a constant map, which we know must happen if $\mathcal{V}(M,v)\not=\emptyset$). 
Now given an element $f\in Maps(M,S^2; f_v)$ we can homotope $f$ so that it is transverse to the north pole $p$. Then $L_f=f^{-1}(p)$ is a link contained in the interior of $M$. Moreover we can use $f$ to give a framing ${\bf f}_f$ to $L_f$. As $f$ homotopes through maps in $Maps(M,S^2; f_v)$, the framed link $(L_f, {\bf f}_f)$ changes by a framed cobordism. Thus to any homotopy class of vector field extending $v$ we can assign a well-defined framed cobordism class of framed links contained in the interior of $M$. The standard proof of the Pontryagin--Thom construction in the closed case easily extends to show this is actually a one-to-one correspondence. This establishes the following relative version of the Pontryagin--Thom construction. 
\begin{lem}\label{relative-tp}
Assume that $\eta$ is a plane field defined along the boundary of $M$ that in some trivialization of $TM$ corresponds to a function that misses the north pole of $S^2$ (as discussed above). There is a one-to-one correspondence between homotopy classes of plane fields on $M$ that extend  $\eta$ on $\partial M$ and the set of framed links in the interior of $M$ up to framed cobordism. \qed
\end{lem}

Let $\mathcal{F}(M)$ denote the group of framed links in the interior of $M$ up to framed cobordism. If $(L,{\bf f})$ is a framed link in the interior of $M$ then $L$ represents a homology class $[L]$, so we can define a homomorphism 
\[
\Phi: \mathcal{F}(M)\to H_1(M;\Z): (L,{\bf{f}})\mapsto [L].
\]
The homomorphism $\Phi$ is clearly surjective. In order to determine the preimage of a homology class we first recall that there is a natural ``intersection pairing" between $H_1(M;\Z)$ and $H_2(M,\partial M;\Z)$. Let $i_*:H_2(M;\Z)\to H_2(M,\partial M;\Z)$ be the map induced from the inclusion $(M,\emptyset)\to (M,\partial M)$. For $h\in H_1(M;\Z)$ set 
\[
D_h =\{h\cdot [\Sigma]: [\Sigma]\in  i_*(H_2(M;\Z))\},
\]
where $h\cdot [\Sigma]$ denotes the intersection pairing between the two homology classes. The set $D_h$ is clearly a subgroup of $\Z$. Let $d(h)$ be the smallest non-negative element in $D_h$. 
\begin{lem}\label{framings}
With the notation as above 
\[
\Phi^{-1}(h)=\Z/d(2h)\Z,
\] 
for any $h\in H_1(M;\Z)$.\qed
\end{lem}
The proof of this lemma is an easy adaptation of the argument for the closed case in \cite{Gompf98}, so the reader is referred there for the proof. 

\begin{remark}
{
It is well known in the closed case that if $(L_\xi,\bf{f}_\xi)$ is a framed link representing a plane field $\xi$ then $2[L_\xi]$ is the Poincar\'e dual of the Euler class of $\xi$. The same reasoning shows in the relative case that $2[L_\xi]$ is the Euler class of $\xi$ relative to $v$.
}
\end{remark}

\subsection{Homotopy classes of plane fields on link complements}
We are now ready to prove Theorems~\ref{thm:legloose} and~\ref{thm:transverseloose}. We begin by observing the following consequence of our discussion above. This result is a ``folk'' theorem that has appeared in the literature, see for example \cite{EliashbergFraser09}, though details of the argument have not appeared. Geiges and, independently, Klukas have also given unpublished proofs of this result using different techniques.  
\begin{lem}\label{tbrdet}
Let $L$ and $L'$ be two null-homologous Legendrian knots in a contact manifold $(M,\xi)$. Let $N$ and $N'$ denote standard neighborhoods of $L$ and $L'$ respectively. If $L$ and $L'$ are topologically isotopic, $\tb(L)=\tb(L')$ and $r(L)=r(L')$ then (after identifying $N'$ with $N$ via any preassigned isotopy and pushing $\xi'$ forward by this isotopy) $\xi|_{M\setminus N}$ is homotopic as a plane field to $\xi|_{M\setminus N'}$ relative to the boundary. 
\end{lem}
Recall if the Euler class of $\xi$ is non-zero then to define the rotation number one needs to specify a homology class for the Seifert surface of $L$. In this case we assume the Seifert surfaces for $L$ and $L'$ are related by the same ambient isotopy that relates $L$ and $L'$.
\begin{proof}
By assumption there is an ambient isotopy of $M$ taking $L'$ to $L$. Pushing $\xi$ forward by this isotopy we have two plane fields $\xi$ and $\xi'$ that agree on a standard neighborhood $N$ of $L=L'$. To prove the theorem it suffices to show $\xi|_{M\setminus N}$ is homotopic, rel boundary, to $\xi'|_{M\setminus N}$. By Lemma~\ref{relative-tp} the homotopy class of these plane fields is determined by the framed links $(L_\xi,{\bf{f}}_\xi)$ and $(L_{\xi'}, {\bf{f}}_{\xi'})$ associated to them by the Pontryagin--Thom construction. According to Lemma~\ref{framings} we need to check that $L_\xi$ and $L_{\xi'}$ represent the same element in $H_1(M\setminus N;\Z)$ and that the framings differ by a multiple of $2d([L_\xi])$. 

When applying Lemma~\ref{relative-tp} and Lemma~\ref{framings} we can use any fixed framing of the tangent bundle that satisfies the hypothesis of Lemma~\ref{relative-tp}. To make the computations below easier we choose a convenient trivialization. We begin by taking the Reeb vector field $v$ to the contact structure $\xi$. After fixing a metric we denote the plane field orthogonal to the vector $v$ by $\xi_v$ (we begin with a metric for which $\xi$ is orthogonal to $v$). At this point we would like to emphasize that we are currently constructing a trivialization of the tangent bundle of $M$ and alterations to $\xi_v$ below do not affect either $\xi$ or $\xi'$. 

Along $L=L'$ choose the unit tangent vector field $u$ to $L$ if $r(L)=r(L')$ is even, otherwise choose $u$ to be the unit tangent vector in $\xi_v$ along $L$ with one negative (clockwise) twist added to it. The vector field $u$ can be extended to the neighborhood $N$ of $L$ (recall $\xi$ and $\xi'$ agree on $N$). We would like to extend $u$ to all of $M$. This is, in general, not possible so we begin by modifying $v$ and $u$. To this end notice that a simple computation reveals that with $X=M\setminus N$ we have $H^2(X,\partial X;\Z)\cong H_1(X;\Z)\cong H_1(M;\Z)\oplus \Z$, with the $\Z$ factor generated by the meridian $\mu$ to $L$. Thus the Euler class of $\xi_v$ relative to $u$ on $\partial X$ is determined by its evaluation on absolute chains in $X\subset M$ (that is determined by the Euler class of $\xi_v$ on $M$) and by the evaluation on the Seifert surface of $L$ in $X$. Since $\xi_v$ is a contact structure we may perform half Lutz twists on transverse curves in $X$. Each such twist changes the Euler class of $\xi_v$ by twice the Poincar\'e dual of the transverse curve. Since $\xi_v$ is oriented the component of the relative Euler class in $H_1(M;\Z)$ is even and by the choice of $u$ above the component of the relative Euler class in $\Z$ is also even. Thus by Lutz twists we can arrange that the Euler class of $\xi_v$, relative to $u$ on $\partial X$, is zero, so $u$ may be extended, as a section of $\xi_v$, over $X$. Now choosing an almost complex structure $J$ on $\xi_v$ we can let $w=Ju$. We can use $-v, u,w$ to trivialize $TM$ and $TX$. (That is the vector field $v$ maps to the south pole of $S^2$.)

Using this trivialization we have the framed links  $(L_\xi,{\bf{f}}_\xi)$ and $(L_{\xi'}, {\bf{f}}_{\xi'})$ associated to $\xi$ and $\xi'$ by the Pontryagin--Thom construction. Moreover, each of these links is also associated to the contact structure $\xi$ on $M$ (since we arranged that $v$ describes $\xi$ in $N$ and it maps to the south pole of $S^2$). Thus the components of $L_\xi$ and $L_{\xi'}$ in $H_1(M;\Z)$ agree. Moreover, if the rotation number of $L$ is even then it is clear that $L_\xi\cap \Sigma=\rot(L)$, where $\Sigma$ is the Seifert surface. Similarly for $L_{\xi'}$ and $\rot(L')$. If the rotation numbers are odd then $L_\xi\cap \Sigma=\rot(L)+1=\rot(L')+1=L_{\xi'}\cap \Sigma$. Thus we see that $L_\xi$ is homologous to $L_{\xi'}$.

Since $\xi$ and $\xi'$ are homotopic as plane fields on $M$ we know that ${\bf{f}}_\xi$ and ${\bf{f}}_{\xi'}$ differ by the divisibility of the image of the Euler class of $\xi$ on $H_2(M;\Z)$. But this is exactly $2d([L_\xi])$ as defined above. Thus $\xi$ and $\xi'$  on $M\setminus N$ are homotopic relative to the boundary. 
\end{proof}

\begin{proof}[Proof of Theorem~\ref{thm:legloose}]
Eliashberg's classification of overtwisted contact structures in \cite{Eliashberg89} says that two contact structures are isotopic as contact structures if and only if they are homotopic as plane fields. Thus if $L$ and $L'$ are two loose Legendrian knots with the same Thurston-Bennequin invariant and rotation number, then Lemma~\ref{tbrdet} implies the complements of $L$ and $L'$ are contactomorphic (rel boundary). The contactomorphism can clearly be extended over the solid torus neighborhood of $L$ and $L'$. Thus $L$ and $L'$ are coarsely equivalent. 

Now given a knot type $\mathcal{K}$ and an overtwisted contact structure $\xi$ there is an overtwisted disk $D$ and a knot $K$ in the knot type $\mathcal{K}$ that is disjoint from $D$. We can $C^0$-approximate $K$ by a Legendrian knot $L$. Thus we may assume that $L$ is disjoint from $D$. Let $U$ and $U'$ be the Legendrian boundary of $D$ with opposite orientations. Notice that $\tb(U)=\tb(U')=0$ and $\rot(U)=-\rot(U')=1$. It is well known, see \cite{EtnyreHonda03}, how the Thurston-Bennequin invariant and rotation numbers behave under connected sums, so we can conclude that 
\[
\tb(L\# U)=\tb(L)+\tb(U)+1=\tb(L)+1 \quad \text{and} \quad \rot(L\# U)=\rot(L)+\rot(U)=\rot(L)+1.
\]
Since connect summing $L$ with $U$ or $U'$ does not change the knot type of $L$ we see that we can find a Legendrian knot in the knot type of $\mathcal{K}$ with Thurston-Bennequin invariant one larger than that of $L$ and with rotation number one larger or one smaller (by connect summing with $U'$). 

We can assume that $D$ is convex and hence there is an embedding of $D^2\times[-1,1]$ such that $D^2\times\{t\}$ is an overtwisted disk for all $t\in[-1,1]$. Thus we have arbitrarily many copies of $U$ and $U'$. By repeated connect summing of $L$ with $U$ and $U'$ or stabilizing $L$ we can change the Thurston-Bennequin invariant to any desired integer. Moreover, it is easy to combine connect summing with $U$ and $U'$ and stabilization to realize any integer, of the appropriate parity, as the rotation number of a Legendrian in the knot type $\mathcal{K}$ without changing the Thurston-Bennequin invariant. This finishes the proof of the theorem. 
 \end{proof}

\begin{proof}[Proof of Theorem~\ref{thm:transverseloose}]
Given two loose transverse knots $T$ and $T'$ in an overtwisted contact manifold $(M,\xi)$ with the same self-linking number, we can choose Legendrian approximations $L$ and $L'$ of $T$ and, respectively, $T'$ such that they have the same Thurston-Bennequin invariant and rotation number (just take any Legendrian approximations of each knot and negatively stabilize one of them if necessary). We can choose $L$ to be in any pre-chosen neighborhood of $T$. Thus we can choose $L$ so that it is loose. Similarly we can assume that $L'$ is loose. From Theorem~\ref{thm:legloose} there is a contactomorphism of $(M,\xi)$ taking $L$ to $L'$. As the transverse push-off of a Legendrian knot is well-defined, we can isotope this contactomorphism through contactomorphisms so that it takes $T$ to $T'$.

Lastly, we can clearly use the Legendrian knots realized in Theorem~\ref{thm:legloose} to realize all the claimed self-linking numbers. 
\end{proof}

\section{Non-loose transverse knots}
In this section we prove our main theorems concerning transverse knots, that is Theorems~\ref{thm:infinite1} through~\ref{thm:tclassify}. We begin with the construction of infinite families of transverse knots with the same classical invariants. 

\begin{proof}[Proof of Theorem~\ref{thm:infinite1}] As in the statement of the theorem let $(B,\pi)$ be a non-trivial open book for the manifold $M$, and $\xi_B$ its corresponding contact structure.
Let $\xi_n$ be the contact structure obtained from $\xi_B$ by adding $n$ full Lutz twists along $B$. Let $B_n$ be the core of the Lutz twist tube in $\xi_n$. Clearly each of the $\xi_n, n>0$, is overtwisted and homotopic, as a plane field, to $\xi_B$ for all $n$. Thus Eliashberg's classification of overtwisted contact structures in \cite{Eliashberg89} implies all the $\xi_n,n>0$, are isotopic contact structures and we denote a representative of this isotopy class by $\xi$. Isotoping all the $\xi_n$ to $\xi$ we can think of $B_n$ as a transverse knot in $\xi$.

We claim that all the $B_n$ are non-loose. Indeed denote the complement of $B_n$ in $(M,\xi)$ by $(C_n, \xi_n')$. To show $(C_n, \xi_n')$ is tight we give a different construction of these contact manifolds. Consider a standard neighborhood  $N(B)$ of $B$ in $(M,\xi_B)$. In particular there is some $a$ such that $N(B)$ is contactomorphic to $\{(\phi, (r,\theta))\in S^1\times \R^2 | r\leq a\}$ with the contact structure $\ker (\phi+r^2\, d\theta)$, where $\phi$ is the coordinate on $S^1$ and $(r,\theta)$ are polar coordinates on $\R^2$. 
Thus the characteristic foliation on $\partial N(B)$ is a linear  (in particular pre-Lagrangain) foliation by lines of some slope $s$. Notice that we can choose numbers $s_n\in [2n\pi,2n\pi +\pi/2]$ such that the manifold $T^2\times [0,s_n]$ with the contact structure $\ker(\sin t\, dx +\cos t\, dy)$, where $t$ is the coordinate on $[0,s_n]$ and $(x,y)$ are coordinates on $T^2$, which we denote by $(T_n,\zeta_n)$, has the following properties: 
\begin{enumerate}
\item the characteristic foliation on $T^2\times\{s_n\}$ is a linear foliation by lines of slope $s$,
\item the characteristic foliation on $T^2\times\{0\}$ is a linear foliation by lines of slope 0,
\item the Giroux torsion of $(T_n,\zeta_n)$ is $n$, and
\item $(T_n,\zeta_n)$ is universally tight.
\end{enumerate}
Items (1) and (2) are obvious, Item (3) is proved in \cite{Giroux00, Kanda97} as is Item (4), though it is also easily checked. 

Let  $(\overline{C}_n, \overline{\xi}_n)$be the manifold obtained by gluing $(\overline{M-N(B)}, \xi|_{\overline{M-N(B)}})$ and  $(T_n,\zeta_n)$ along their boundaries.   It is clear that the complement $C_n$ of $B_n$ in $(M,\xi)$ is the interior of   $(\overline{C}_n, \overline{\xi}_n)$.
Since it is well known that the complement of the binding of an open book is universally tight, see \cite{EtnyreVela-Vick10}, it is clear $(\overline{M-N(B)}, \xi|_{\overline{M-N(B)}})$ is universally tight. We now recall that Colin's gluing theorem, \cite{Colin99}, says that gluing two universally tight contact structure along a pre-Lagrangian incompressible torus results in a universally tight contact structure. Thus  $(\overline{C}_n, \overline{\xi}_n)$, and hence  $(C_n,\xi'_n)$,  is universally tight. 

To show infinitely many of the $B_n$ are distinct we need the following observation that follows immediately from Lemma~\ref{lem:torsioncompute}, which is proven below.
\begin{lem}\label{lem:finite}
Let $[T]$ be the isotopy class of tori in $C_n\subset \overline{C}_n$ parallel to the boundary. 
The Giroux torsion of $(\overline{C}_n, \overline{\xi}_n)$ in the isotopy class $[T]$ is finite:  
\[
\tor ((\overline{C}_n, \overline{\xi}_n),[T])<\infty.
\]
Hence $\tor ((C_n,\xi|_{C_n}),[T])$ is also finite.
\end{lem}
Since it is clear the Giroux torsion of $(C_n,\xi|_{C_n})$  in the isotopy class of $[T]$ is greater than or equal to $n$ and all the $C_n$'s have finite torsion, we can conclude that infinitely many of the $C_n$ are distinct and hence infinitely many of the $B_n$ are distinct. 

Finally we notice that $sl(B_n)=-\chi(B)$ for all $n$. Indeed notice that a Seifert surface for $B_n$ can be built by taking a Seifert surface for $B$ (that is a page of the open book) and extending it by the annulus $\{p\}\times S^1\times (0,s_n]$ in $T_n$. As we have an explicit expression for the contact structure in a Lutz tube it is easy to see that this annulus can be perturbed to have no singularities. As the self-linking number of $B_n$ can be computed from the singularities of a Seifert surface for $B_n$ we see the all the self-linking numbers must agree with $sl(B)$ which is well known to be $-\chi(B)$.
\end{proof}

\begin{proof}[Proof of Theorem~\ref{thm:infinite2}] The argument here is almost identical to the one given above for Theorem~\ref{thm:infinite1}. The only difference being that one begins by performing a half Lutz twist on $B$ and notices that this changes the self-linking number of $B$ from $-\chi(B)$ to $\chi(B)$. 
\end{proof}

We now turn to the proof of the above lemma. Using the notation established in the previous proof, we will actually compute the Giroux torsion of $(\overline{C}_n, \overline{\xi}_n)$ in the isotopy class of $[T]$, which by Lemma~\ref{lem:torsion2} is the same as $\tor ((C_n,\xi|_{C_n}),[T])$.

For each $n$, the manifold $\overline{C}_n$ is (canonically up to isotopy) diffeomorphic to $\overline{M-N(B)}$ since $\overline{C}_n\setminus\overline{M-N(B)}$ is $T^2\times [0,1]$. We denote this common diffeomorphism type as $\overline{C}$ and think of the contact structures $\overline{\xi}_n$ constructed above as contact structures on $\overline{C}$. 

\begin{lem}\label{lem:torsioncompute}
With notation as above
\[
\tor((\overline{C}, \overline{\xi}_n),[T])=n.
\]
\end{lem}
Notice that Lemma~\ref{lem:finite} immediately follows from this result. The proof of this lemma is inspired by Proposition~4.6 in \cite{HondaKazezMatic02}. Slightly modifying our proof above, we could cite this result in place of Lemma~\ref{lem:torsioncompute} to prove the above theorem, but to prove our results below we need to identify the actual Giroux torsion which takes more work. 

We begin by defining an invariant and establishing a few properties.  Let $\Sigma$ be a page of the open book $(B,\pi)$ in $M$. This gives a properly embedded surface, also denoted $\Sigma$, in $\overline{C}$.  Let $\gamma$ be a properly embedded non-separating arc in $\Sigma$. Considering the contact structure $\overline{\xi}_n$ we know that the characteristic foliation of $\partial \overline{C}$ consists of meridional curves. We can assume that $\partial \Sigma$ is (positively) transverse to this foliation. Let $\mathcal{L}(\gamma)$ be the set of all Legendrian arcs $\gamma'$ embedded in $(\overline{C},\overline{\xi}_n)$ satisfying 
\begin{enumerate}
\item the arc $\gamma'$ lies on some convex surface $\Sigma'$ with (positively) transverse boundary in $\partial\overline{C}$,   and 
\item  there is a proper isotopy of $\Sigma'$ to $\Sigma$ through surfaces with transverse boundary on $\partial\overline{C}$ that takes $\gamma'$ to $\gamma$.  
\end{enumerate}
Set 
\[
mt(\gamma)= \max_{\gamma'\in \mathcal{L}(\gamma)} \{tw(\gamma',\Sigma')\}
\]
where $tw(\gamma',\Sigma')$ is the twisting of the contact planes along $\gamma'$ with respect to the framing given to $\gamma'$ by $\Sigma'$. 
We can show that this invariant is bounded by the Giroux torsion of $\eta_n$.
\begin{lem}\label{lem:bound}
With the notation as above we have the inequality
\[
mt(\gamma)\leq -2\tor((\overline{C}, \overline{\xi}_n),[T]).
\]
\end{lem}
\noindent
We can also compute the invariant. 
\begin{lem}\label{lem:compute}
In the contact manifold $(\overline{C},\overline{\xi}_n)$ we have
$mt(\gamma)=-2n$. \qed
\end{lem}
Before proving Lemmas~\ref{lem:bound} and~\ref{lem:compute} we note that the previous two lemmas immediately yield our main lemma above.
\begin{proof}[Proof of Lemma~\ref{lem:torsioncompute}]
Lemmas~\ref{lem:bound} and~\ref{lem:compute} allow us to conclude that $\tor((\overline{C}, \eta_n),[T])\leq n$. But from construction we know $\tor((\overline{C}, \eta_n),[T])\geq n$. 
\end{proof}

We also observe that the proof of Lemma~\ref{lem:compute} is just as easy.
\begin{proof}[Proof of Lemma~\ref{lem:compute}]
As we know there is an embedding of $(T^2\times[0,1], \zeta_n)$ into $(\overline{C}, \overline{\xi}_n)$ we can use Lemma~\ref{lem:bound} to see that $mt(\gamma)\leq -2n$. But we can explicitly construct a curve $\gamma'$ on a surface $\Sigma'$ with $tw(\gamma',\Sigma')=-2n$. The surface $\Sigma'$ is constructed as at the end of the proof of Theorem~\ref{thm:infinite1}; that is, it is the union of a Seifert surface for $B$ and an annulus in $T_n$. We can Legendrian realize a curve on the Seifert surface for $B$  (note one must be careful at this point as the Seifert surface has transverse boundary) and using an explicit model for $T_n$ this curve can be extended across the annulus so that it reprsents the isotopy class of $\gamma$ and has $tw(\gamma',\Sigma')=-2n$.
\end{proof}

We now turn to the proof of Lemma~\ref{lem:bound}.
\begin{proof}[Proof of Lemma~\ref{lem:bound}]
Suppose we have a contact embedding of $\phi: (T^2\times[0,1], \zeta_k)\to (\overline{C}, \overline{\xi}_n)$ in the isotopy class of $[T]$. We need to show that $mt(\gamma)\leq -2k$. To this end take any $\gamma'\in \mathcal{L}(\gamma)$ and let $\Sigma'$ be the convex surface containing $\gamma'$. 

Notice that there are $2k$ tori in the image of $\phi$ whose characteristic foliations consist of leaves isotopic to meridians. On these tori a curve isotopic to $\partial \Sigma$ can be made transverse to the foliation. Half the time it will be positively transverse the other half it will be negatively transverse. It is clear that we can choose one of the tori $T$ where the curve is positively transverse and cobounds with $\partial \overline{C}$ a manifold $A$ such that $(A,(\overline{\xi}_n)|A)$ is contactomorphic to  $(T^2\times[0,1], \zeta_k)$. We now set $D=\overline{C}\setminus A$.

We glue two copies of $\overline{C}$ together by the diffeomorphism $\phi$ of $\partial\overline{C}$ that preserves the meridian (and its orientation) and reverses the orientation on the longitude. Denote the resulting manifold $P=\overline{C}\cup_\phi \overline{C}$. Since $\phi$ preserves the characteristic foliation we clearly have a contact structure $\eta_n$ induced on $P$ from $\overline{\xi}_n$ on $\overline{C}$, we note however that with our choice of gluing map we have $\overline{\xi}_n$ on one of the copies of $\overline{C}$ and $-\overline{\xi}_n$ on the other. Let $\widetilde{P}$ be the infinite cyclic cover of $P$ that unwinds the meridian of $\overline{C}\subset P$. If we denote the double of $\Sigma$ by $S$, then it is clear that $\widetilde{P}$ is $S\times\R$.
\begin{claim}
The pullback  $\widetilde{\eta}_n$ of $\eta_n$ is an $\R$-invariant tight contact structure and $S$ can be assume to be convex with dividing set $\Gamma_S$ containing at least $4k-1$ curves parallel to $\partial \Sigma\subset S$. 
\end{claim}

\noindent
Given this we can finish the proof of the lemma as follows. Let $c$ be the double of $\gamma'$ sitting on $S$. We clearly have that the twisting of the contact structure $\eta_n$ along $c$ measured with respect to $S$, which we denote by $t(c)$, is twice the twisting along of $\gamma'$. So if we can show $t(c)\leq -4k+1$ then the lemma will be established (notice that since $t(c)=2mt(\gamma)$ we know $t(c)$ is even so this inequlaity actually implies $t(c)\leq -4k$). We can lift $c$ to $\widetilde{P}$ and notice that its twisting is unchanged (since the covering map regularly covers a small neighborhood of $c$).

We recall that Giroux's ``semi-local Bennequin inequality'' says the following.
\begin{thm}[Giroux 2001, \cite{Giroux01}]
Let $\xi$ be an $\R$ invariant tight contact structure on $S\times \R$ where $S$ is a closed orientable surface of genus greater than zero. Let $\Gamma$ be the dividing set on $S\times \{0\}$ (which is clearly convex) and $C$ an essential simple closed curve in $S\times\{0\}$. Then for any Legendrian curve $L$ smoothly isotopic to $C$
\[
\text{tw}(L,\mathcal{F})\leq -\frac 12 (\Gamma\cdot C),
\]
where $\text{tw}(L,\mathcal{F})$ denotes the twisting of the contact planes along $L$ measured with respect to the framing $\mathcal{F}$  given by $S\times\{0\}$ and  $\Gamma\cdot C$ denotes the minimal geometric intersection between curves isotopic to $\Gamma$ and $C$.
\end{thm}
\noindent
Thus any curve $c'$ isotopic to $c$ satisfies $t(c')\leq -\frac 12 (\Gamma_S \cdot c')$, where $\Gamma_S \cdot c'$ denotes the minimal geometric intersection of a curve isotopic to $c'$ with $\Gamma_S$. Since $c'\cdot \partial \Sigma=2$  and $\Gamma_S$ contains $4k-1$ curves parallel to $\partial \Sigma$ it is clear that $\Gamma_S \cdot c'\geq 8k-2$ from which the lemma follows. 

We now establish the claim. Let $D_c$ be the manifold $D$ with the leaves in the characteristic foliation on $\partial D$ collapsed to points (topologically we are just Dehn filling along the meridional slope). The contact structure $(\overline{\xi}_n)|_D$ descends to give a contact structure $\eta$ on $D_c$ and the image of $\partial D$ in $D_c$ is a positive transverse curve $B'$.  One easily sees that $B'$ is the binding of an open book for $D_c$ and the page of this open book is diffeomorphic to $\Sigma$. We notice that $B'$ is the image of a positive transverse curve $\widetilde{B}'$ in $D\subset \overline{C}$. This curve cobounds an annulus in $A$ with a positive transverse curve $\widetilde{B}$ in $\partial \overline{C}$ that descends to the binding $B$ of the open book for $M$ when the characteristic foliation on $\overline{C}$ is collapsed to form $M$. We know from construction that $\sl(B)=-\chi(\Sigma)$. Thus one may easily conclude the same for $\widetilde{B}, \widetilde{B}'$ and $B'$. From Theorem~\ref{thm:maxclass} we now see that $\eta$ is quasi-compatible with $B'$. Thus there is a contact vector field $v$ on $D_c-B'$ that is transverse to the pages of the open book. One may easily check (by considering a local model for $B'$) that $v$ may be assumed to be meridional (that is, its flow lines are meridians) inside a neighborhood of $B'$ but outside a smaller neighborhood of $B'$. Thus we may alter $v$ on $D_c-B'=D-\partial D$ so that in a neighborhood of $B'$ minus $B'$ the orbits of $v$ are meridians, and hence we can then easily be  extended over all of $D$. Moreover, we can find a contact vector field on $A$ that also has meridional flow lines and then use it to extend $v$ to a contact vector field on all of $\overline{C}$ that is transverse to $\Sigma$ (and all the pages of the open book). It is clear that the dividing set on $\Sigma$ induced by $v$ contains at least $2k$ closed curves parallel to $\partial \Sigma$, one of them being $\partial \Sigma$. The manifold $P$ is obtained by gluing together two copies of $\overline{C}$. Observing that the gluing map preserves $v$ we can get a contact vector field on $(P,\eta_n)$ that is transverse to $S$ (and all the fibers in the fibration of $P$ over $S^1$). Thus on the $\Z$-cover $\widetilde{P}=S\times \R$ of $P$ we can lift $v$ to a vector field $\widetilde{v}$ that preserves $\widetilde{\eta}_n$ and is transverse to $S\times\{t\}$ for all $t\in \R$. We can use the flow of $\widetilde{v}$ to identify $\widetilde{P}$ with $S\times \R$ so that the contact structure is $\R$-invariant. From our observation about the dividing set on $\Sigma$ we see that $S$ has at least $4k-1$ dividing curves parallel to $\partial \Sigma$. (The minus one comes from the fact that one of the dividing curves on $\Sigma$ was $\partial \Sigma$ so when the two copies of $\Sigma$ are glued together, two of the dividing curves are identified.)
\end{proof}

\begin{proof}[Proof of Theorem~\ref{GeneralInfinite}]
Let $K$ be a null-homologous knot type in the irreducible manifold $M$, with a Seifert surface $\Sigma$ of genus $g>0$, where $g$ is the genus of $K$. Set $C=\overline{M-N}$ where $N$ is a solid torus neighborhood of $K$. Let $\Gamma$ be the union of two embedded meridional curves on $\partial C$. It is easy to check that $(C,\Gamma)$ is a taut sutured manifold and thus according to Theorem~1.1 in \cite{HondaKazezMatic02} there is a universally tight contact structure $\xi$ on $C$ for which $\partial C$ is convex with dividing curves $\Gamma$. Moreover, in the proof of Theorem~1.1 in \cite{HondaKazezMatic02} we see that one may assume that $\partial \Sigma$ is Legendrian and $\Sigma$ is convex with a single dividing curve that is boundary parallel. 

Now let $T=T^2\times [0,\infty)$ with the contact structure $\xi_T=\ker (\cos z \, dx + \sin z\, dy)$. Using the convex version of Colin's gluing criterion, Theorem 5.7 in \cite{HondaKazezMatic02} we see that we can glue $(C,\xi)$ and either $(T,\xi_T)$ or $(T,-\xi_T)$ together to get a universally tight contact structure $\xi$ on a manifold $C'$ diffeomorphic to $C$. 

There is a sequence of disjoint tori $T_n, n\in \N$, in $T\subset C'$ that have linear characteristic foliation each leaf of which is a meridional curve in $C'$. Moreover we can arrange that $T_n$ and $T_{n+1}$ cobound a $T^2\times[0,1]$ with Giroux torsion 1.  Let $C_n$ be the compact component of $C'\setminus T_n$. Notice that if each leaf of the characteristic foliation of $\partial C_n=T_n$ is collapsed to a point (that is, topologically we Dehn fill $C_n$) we get a manifold diffeomorphic to $M$ (and this diffeomorphism is canonical up to isotopy). Moreover there is a neighborhood $U$ of $T_n$ in $C_n$  and a neighborhood $V$ of $K=S^1\times\{(0,0)\}$ in $S^1\times D^2$, with contact structure $d\phi+r^2\, d\theta$, such that $U-T_n$ and $V-K$ are contactomorphic. The collapsing process going from $C_n$ to $M$ can be thought of as removing $U$ and replacing it with $V$. Thus we see that $C_n$ induces a contact structure $\xi_n$ on $M$ and in that contact structure there is a knot $K_n$ such that $M-K_n$ is contactomorphic to $C_n-T_n$. Notice that all of the $\xi_n$ are overtwisted and $\xi_n$ is obtained from $\xi_{n-1}$ by a full Lutz twist on $K_n$. Thus by Eliashberg's classification of overtwisted contact structures we know all the $\xi_n$ are isotopic to a fixed overtwisted contact structure which we denote $\eta$. As above each $K_n$ gives a transverse knot, still denoted $K_n$, in $\eta$. Each $K_n$ is clearly non-loose. Moreover the surface $\Sigma$ above can be extended by an annulus in $T$ so that in $M$ it gives a Seifert surface for $K_n$. It is clear from construction that $\Sigma$ is convex with dividing curves parallel to the boundary. One can choose the $T_n$ so that $\Sigma$ will always have an even number of dividing curves or an odd number of dividing curves. If we choose the former then it is clear that all the $K_n$ have self-linking number $-\chi(\Sigma)=-\chi(K)$. If we choose the latter then all the $K_n$ have $\sl(K_n)=\chi(K)$. Also note that going between the former and the latter amounts to doing a half-Luts twist on the knots $K_n$. Thus we see that the contact structure where the $K_n$ have self-linking number $-\chi(K)$ and the contact structure where they have self-linking number $\chi(K)$ differ by a half-Lutz twist and hence are not contactomorphic. 

Thus we will be done with the theorem once we see that all the $K_n$ are non contactomorphic. If we perturb the boundary of $C_n$ so that it is convex with two dividing curve then the resulting contact structures are the same as the ones constructed in Proposition~4.2 of \cite{HondaKazezMatic02}. In Proposition~4.6 of that paper  it is also shown that all these contact structures are not contactomorphic.  Now if there was a contactomorphism of $(M,\eta)$ taking $K_n$ to $K_m$ then $C_n-T_n$ would be contactomorphic to $C_m-T_m$ Denote the contactomorphism by $\phi$.  Let $B$ be a torus in $C_n-T_n$ with linear characteristic foliation of slope 0 (that is the leaves are null-homologous in $C_n$). Let $(C_n-T_n)\setminus B=P\cup Q$ with $Q$ the non-compact component. We can assume $B$ was chosen so that $Q$ is minimally twisting (that is there are no convex tori with negative slope in $Q$). The torus $\phi(B)$ breaks $C_m-T_m$ into two similar such pieces $P'$ and $Q'$ and $\phi$ gives a contactomorphism from $P$ to $P'$. By adding the appropriate basic slice to both $P$ and $P'$ we can extend $\phi$ to a contactomorphism of the contact structures constructed in Proposition~4.2 of \cite{HondaKazezMatic02}. This would contradict Proposition~4.6 of that paper unless $m=n$.
\end{proof}

We now turn to the classification of transverse knots with maximal self-linking in the hyperbolic knot type of the binding of an open book supporting the given contact structure. 

\begin{proof}[Proof of Theorem~\ref{thm:tclassify}]
We are given an open book decomposition $(B,\pi)$ with connected binding. Since $B$ is a hyperbolic knot the monodromy $\phi$ of the open book is pseudo-Anosov. Let $\xi$ be the contact structure obtained from $\xi_B$ by performing a full lutz twist on $B$. Clearly $\xi$ is overtwisted and in the same homotopy class of plane field as $\xi_B$, so if $\xi_B$ is overtwisted then $\xi$ and $\xi_B$ are isotopic contact structures. 

We are trying to determine the set $\mathcal{T}_{-\chi(B)}(B)$. Since any knot type has a loose knot with any odd self-linking number there is clearly a loose knot $K_*$ in $\mathcal{T}_{-\chi(B)}(B)$. Moreover, Theorem~\ref{thm:transverseloose} says that this is the unique loose knot with given knot type and self-linking number so $K_*$ is the only loose knot in $\mathcal{T}_{-\chi(B)}(B)$. Since loose and non-loose knots are not contactomorphic we are left to classify the non-loose knots in $\mathcal{T}_{-\chi(B)}(B)$.

If $K$ is a non-loose knot in $\mathcal{T}_{-\chi(B)}(B)$ then Theorem~\ref{thm:maxclass}
implies that $\xi$ is obtained from $\xi_B$ by adding Giroux torsion along tori which are incompressible in the complement of $B$. Since the monodromy $\phi$ is pseudo-Anosov the only such tori are isotopic to the boundary of a neighborhood of $B$. Adding Giroux torsion along this torus is equivalent to performing some number of full Lutz twists along $B$ in $\xi_B$. Thus $K$ is clearly one of the $B_n$ constructed in the proof of Theorem~\ref{thm:infinite1}. Notice that $B_0$, the knot obtained from $B$ by doing no Lutz twists, is a transverse knot in $\xi$ if and only if $\xi_B$ is overtwisted to begin with. Thus we see that any non-loose knot in $\mathcal{T}_{-\chi(B)}(B)$ is contactomorphic to one of the knots $\{B_n\}_{n\in A}$ where $A$ is the indexing set in the statement of the theorem. Finally Lemma~\ref{lem:torsioncompute} guarantees that $B_n$ and $B_m$ are not contactomorphic if $n\not=m$ since they have different contact structures on their complement.

Renaming $B_n$ to $K_n$,  we have now classified the transverse knots in $\mathcal{T}_{-\chi(B)}(B)$ and  established the first bullet point in the theorem. The second bullet point follows from \cite{Vela-Vick09} since all the knots except $K_0$ have either an overtwisted disk or Giroux torsion in their complement. 

To establish the last bullet point in the theorem we notice that each $K_n$ for $n>0$ has a neighborhood contactomorphic to $S^1\times D^2=\{(\phi, r, \theta): r\leq 2\pi+\epsilon\}$ for some $\epsilon$, with the contact structure $\ker (\cos r \, d\phi + r\sin r\, d\theta)$. Thus we can choose a convex torus $T$ outside of the solid torus $\{r\leq 2\pi\}$ with two dividing curves inducing the framing $-m$ for some $m$ (where we use the page of the open books to define the 0 framing). Thus we can find a neighborhood $N$ of $K_n$ that breaks into two pieces $A=T^2\times[0,1]$ and $N'=S^1\times D^2$, with the following properties. The contact structure on $A$ is not invariant in the $[0,1]$ direction and the boundary of $A$ is convex with each boundary component having two dividing curves inducing the framing $-m$. The contact structure on the solid torus $N'$ is minimally twisting and the boundary of $N'$ is convex with two dividing curves inducing the framing $-m$. From \cite{EtnyreHonda01b, Honda00a}, we know $N'$ is the standard neighborhood of a Legendrian curve $L$ whose positive (if $L$ is oriented in the same way that $K_n$ is oriented) transverse push-off is $K_n$. If we stabilize $L$ positively then its positive transverse push-off is a transversely isotopic to the stabilization of $K_n$, see \cite{EtnyreHonda01b}. Let $N''$ be a neighborhood of the stabilized Legendrian inside of $N'$. From \cite{EtnyreHonda01b} we know that we can write $N'$ as the union of $A'=T^2\times[0,1]$ and $N''$ where $A'$ is a basic slice. The sign of this basic slice depends on the stabilization of $L$ that we perform. One may easily check that all the basic slices in any decomposition of $A$ into basic slices all have the same sign and that sign is opposite the one associated to $A'$ if we positively stabilize $L$. Thus the contact structure on $A\cup A'$ is a non-minimally twisting contact structure made from basic slices with different signs. Such a contact structure must be overtwisted, \cite{Honda00a}. Since the transverse push-off of the stabilized $L$ is contained in $N''$ we see that the complement of this transverse knots is overtwisted since it contains $A\cup A'$. Thus the complement of a stabilization of $K_n$ is overtwisted and we see that a stabilization of $K_n$, for $n>0$, is loose. Theorem~\ref{thm:transverseloose} then guarantees all the stabilizations of the $K_n, n>0$ are contactomorphic to the stabilization of $K_*$. We note that it is not clear if the stabilizations of $K_0$, if it exists, are loose or not. 

Finally, If $\xi'$ is some overtwisted contact structure and $T$ is a transverse knot in $\mathcal{T}_{-\chi(B)}(B)$ that is non-loose then Theorem~\ref{thm:maxclass} implies that $\xi'$ is obtained from $\xi_B$ by adding some number of full Lutz twists along $B$ (notice that the only incompressible tori in the complement of $B$ are boundary parallel tori). Thus if $\xi'$ is not so obtained then $\mathcal{T}_{\chi(B)}(B)=\{K_*\}$.
\end{proof}

\section{Non-loose Legendrian knots}
Using well know facts concerning the relation between Legendrian and transverse knots, mostly reviewed in Subsection~\ref{revLegTrans} above, we can upgrade the coarse classification of transverse knots from Theorem~\ref{thm:tclassify} to Legendrian knots.

\begin{proof}[Proof of Theorem~\ref{thm:lclassify}]
By Theorem~\ref{thm:tclassify} we know that $\mathcal{T}_{-\chi(B)}(B) = \{K_*\}\cup \{K_i\}_{i\in A}$, where $A=\N$ if $\xi_B$ is tight and $A=\N\cup\{0\}$ if not. For any $i\in\N$ the transverse knot $K_i$ was the core of a full Lutz twist. Thus it has a neighborhood $N_i$ that is contactomorphic to $S^1\times D^2_{2\pi}$, with the contact structure $\ker (\cos r\, d\phi + r\sin r\, d\theta)$, where $D^2_a$ is the disk of radius $a$. There is an infinite sequence of radii $r_j, j$ an integer, such that $S^1\times D^2_{r_j}$ has a linear characteristic foliation of slope $\frac{1}{j}$. Let $L_{j,i}$ be a leaf in this foliation. Lemma~\ref{legapprox} says that $K_i$ is the transverse push-off of $L_{j,i}$. Since the contact framing of $L_{j,i}$ is the same as the framing coming form the torus $\partial (S^1\times D^2_{r_j})$ it sits on, we see that $\tb(L_{j,i})=j$ and hence $\rot(L_{j,i})=\chi(B)+j$. If $\xi_B$ is overtwisted then we can similarly construct Legendrian approximations $L_{i,0}$ of $K_0$ for all $i$ less than some $m$, where $m$ is either an integer or $\infty$. It is clear from Proposition~\ref{nltnll} that all the $L_{i,j}$ constructed here are non-loose. Moreover, they are all distinct as they either have different transverse push-offs or they have different Thurston-Bennequin invariants. 

All the bullets points in the theorem, except the last, follows from the corresponding statements for the transverse knots $K_i$. The last bullet point follows from Lemma~\ref{legapprox}.

Since there is a unique loose Legendrian knot with given invariants we are left to show that, up to contactomorphism, any non-loose Legendrian knot in $\mathcal{L}_{\chi(B)+n,n}(B)$ is one of the ones constructed above.

Let $L$ be a non-loose Legendrian knot in $\mathcal{L}_{\chi(B)+n,n}(B)$. We begin by assuming that $\tb(L)=-k\leq 0$ and hence $\rot(L)= \chi(B)-k$. Let $N$ be a standard neighborhood of $L$ with convex boundary. We assume the characteristic foliation on $\partial N$ has two closed leaves and all leaves are transverse to a ruling of $\partial N$ by longitudes. Let $\Sigma$ be a fiber in the fibration of $M\setminus N$. From our set up $\partial \Sigma$ is a transverse curve $T\subset \partial N$. Moreover it is easy to see that $T$ is the transverse push-off of $L$ and $\sl(T)=-\chi(B)=-\chi(\Sigma)$. 

By Lemma~3.3 in \cite{EtnyreVanHornMorris10} we can isotope $\Sigma$ so that it has a Morse-Smale characteristic foliation and no negative singular points. Thus $\Sigma$ is convex and dividing curves are disjoint from $\partial \Sigma$. From the proof of Lemma~3.4 in \cite{EtnyreVanHornMorris10} we see that $\Sigma$ may be isotoped so that the dividing curves $\Gamma_\Sigma$ on $\Sigma$ are invariant under the monodromy of the open book. As the monodromy is pseudo-Anosov the only dividing curves $\Sigma$ can have are ones parallel to $\partial \Sigma$. Assume that $\Sigma$ has been isotoped (keeping $\partial \Sigma$ transverse and contained in $\partial N$) to minimize the number of dividing curves. 

We show that the contact structure on the complement of $N$ is determined by the dividing curves on $\Sigma$. 
Taking a neighborhood $N'$ of $N\cup \Sigma$ and rounding corners we see that $\partial N'$ is obtained by gluing together two copies of $\Sigma$, denoted $\Sigma_1$ and $\Sigma_2$,  and an annulus $A$. We can assume that $\partial N'$ is convex and that its dividing curves consist of a copy of $\Gamma_\Sigma$ on each $\Sigma_i$ and one curve in the center of $A$, see \cite{EtnyreVanHornMorris10}. Since $\Sigma$ is a page of an open book for $M$ we see that $M\setminus N'$ is a handlebody. Moreover, applying Lemma~3.5 in \cite{EtnyreVanHornMorris10} we see that the minimality of the number of components of $\Gamma_\Sigma$ implies that the dividing curves on compressing disks are uniquely determined by $\Gamma_\Sigma$. Thus the contact structure on $M\setminus N'$ is determined by the number of curves in $\Gamma_\Sigma$.  Moreover, since the difference between $M\setminus N'$ and $M\setminus N$ is a neighborhood of $\Sigma$, a similar argument says that $M\setminus N$ is determined by the number of curves in $\Gamma_\Sigma$. One may easily check that $L_{-k,n}$ has a  convex surface with the same configuration of dividing curves in the complement of a standard neighborhood, where $2n$ is the number of components in $\Gamma_\Sigma$. Thus $L$ is contactomorphic to $L_{-k,n}$.

If $\tb(L)=k>0$, and hence $\rot(L)=\chi(B)+k$, then we can proceed as above except now the ruling longitudinal curves on $\partial N$, oriented in the same direction as $L$, are negatively transverse curves. Thus $\partial \Sigma$ is a negatively transverse curve $T$ on $\partial N$. We can identify $N$ as a neighborhood of a transverse curve $T'$ that is a positive transverse push-off of $L$, and hence has $\sl(T')=-\chi(B)$. The curves $T$ and $T'$ cobound an annulus $A$ whose characteristic foliation as a single closed leaf $L'$. The Legendrian $L'$ is clearly topologically isotopic to $L$ and has Thurston-Bennequin number 0. Moreover, $T'$ is its positive transverse push-off and $T$ its negative push-off. From the first fact we see $\rot(L')=\chi(B)$ and from the second fact we see $\sl(T)=\chi(B)$. 

Notice that if we orient $T$ in the same direction as $B$ and $\Sigma$ so that it has oriented boundary $T$, then the characteristic foliation on $\Sigma$ points in along $T$ (because $T$ will be a negatively transverse to the contact structure). By the proof of Lemma~3.3 in \cite{EtnyreVanHornMorris10} we can isotope $\Sigma$ so that it has a Morse-Smale characteristic foliation and no positive singular points. We can now argue as above to conclude that $L$ is contactomorphic to $L_{k,n}$ for some $n$.
\end{proof}

\begin{proof}[Proof of Theorem~\ref{thm:linvertableclassify}]
The hypothesis of the theorem says that $L$ and $-L$ are in the same knot type, up to diffeomorphism. Thus if $L$ is in $\mathcal{L}_{\chi(B)+n, n}(B)$ then $-L$ is in $\mathcal{L}_{-\chi(B)-n, n}(B)$, so the result follows from Theorem~\ref{thm:lclassify}. Moreover, all the knots claimed to be in $\mathcal{L}_{-\chi(B)-n,n}(B)$  can be obtained knots in $\mathcal{L}_{\chi(B)+n, n}(B)$ by reversing the orientation on the knot.  
\end{proof}

Finally we establishes our result about stabilizing Legendrian knots with large Thurston-Bennequin invariant. 
\begin{proof}[Proof of Proposition~\ref{prop:stabilize}]
We  assume that $-B$ is not isotopic to $B$. This condition implies there is no diffeomorphism taking $B$ to $-B$ that is isotopic to the identity. Thus a contactomorphism that is smoothly isotopic to the identity and takes a Legendrian knot $L$ in $\mathcal{L}(B)$ to itself must preserve the orientation on $L$ and hence on any surface with boundary on $L$. So we can conclude if the contactomorphism is also co-orientation preserving then it preserves the sense of a stabilization. By this we mean that if $L=S_+(L')$ and $\phi$ is such a contactomorphism of $\xi$ then $\phi(L)=S_+(\phi(L'))$.  

Now suppose that $L\in\mathcal{L}(B)$ and $\tb(L)- \rot(L)> -\chi(B)$. Set $n=\frac 12(-\chi(B)-\tb(L))$. Notice that $\tb(S_+^n(L))-\rot(S_+^n(L))=-\chi(B)$. Thus by Theorem~\ref{thm:lclassify} we know there is a contactomorphism $\phi$ sending $S_+^n(L)$ to one of the Legendrian knots described in the theorem (one may also check that $\phi$ preserves the co-orientation of the contact structure), and from the observation above $S_+^n\phi(L)$ is isotopic to one of the knots described in the theorem. If it is isotopic to any knot other than $L_*$, that is the loose knot, then any number of negative stabilization will stay non-loose. It is clear, however, that there is some $m$ such that $S_-^m(\phi(L))$ cannot be non-loose since it does not satisfy $-|\tb| +|\rot|\leq -\chi(B)$. Thus $S_-^m(S_+^n(\phi(L)))$ cannot be non-loose resulting in a contradiction unless $S_+^n(\phi(L))$ is already loose, which of course implies $S_+^n(L)$ is loose.  
\end{proof}

\def\cprime{$'$} \def\cprime{$'$}

\end{document}